\documentclass[12pt]{amsart}
\usepackage[utf8]{inputenc}
\usepackage{url}
\usepackage{hyperref}
\usepackage{amsmath,amssymb}
\usepackage{amsfonts}
\usepackage{amsthm}
\usepackage{xcolor}
\usepackage{tikz}
\usepackage{tkz-graph}
\usepackage{tkz-berge}
\newtheorem{thm}{Theorem}
\newtheorem{lem}{Lemma}

\newtheorem{prop}[thm]{Proposition}
\newtheorem{cor}[thm]{Corollary}
\newtheorem{claim}[thm]{Claim}
\newtheorem{ex}[thm]{Example}
\newtheorem{rem}[thm]{Remark}
\calclayout
\def\Den#1{{\mathcal D}_{#1}}
\def\num#1{\operatorname{num}\left(#1\right)}
\def\den#1{\operatorname{den}\left(#1\right)}
\def\Yang#1#2{{{\operatorname Y}\kern-0.35em{\operatorname B}}^{#1}_{#2}}

\def\spectreY#1{{\operatorname{spec}}_{#1}}
\def\spectre#1{{\widehat{\operatorname{spec}}}_{#1}}
\usepackage{amsmath, amssymb}
\def\refeq#1{(\ref{#1})}

\def\algo{\mathtt{algo}}
\def\triv{\mathtt{triv}}
\def\jump{\mathtt{jump}}
\def\jumppath{\mathfrak{jump}}
\def\djumppath{{\mathfrak{jump}^\dag}}
\def\jumpop{\operatorname{jump}}
\def\opt{\mathtt{opt}}
\def\lcm{\operatorname{lcm}}
\def\gcd{\operatorname{gcd}}
\def\staircase{\mathtt{staircase}}
\def\Qsc{\mathtt{Qsc}}
\def\rai{\mathfrak{raise}}
\def\pathup{\mathfrak{up}}
\def\addstep{\mathfrak{add\_step}}
\def\raiseop{\operatorname{raise}}
\def\upop{\operatorname{up}}
\def\addop{\operatorname{add\_step}}
\def\affineop{{\operatorname{A}}}

\def\affinepath{\Phi}
\def\vx{{\mathtt{X}}}
\def\vy{{\mathtt{Y}}}
\def\vhy{{\widehat{\mathtt{Y}}}}
\def\setspecY{\mathfrak{Spec}}
\def\setspec{\widehat{\mathfrak{Spec}}}
\def\std#1{\operatorname{std}(#1)}
\def\affineopvec{\Lambda}
\def\i{\mathbf i}
\title{Hunting the poles in the staircases}
\author{Christophe Carré}
\address{GR$^2$IF. Université de Rouen-Normandie. Technopôle du Madrillet
BP 12 Avenue de l'Université
F-76801 Saint-Étienne-du-Rouvray Cedex}
\email{christophe.carre@univ-rouen.fr}
\author{Ulysse Goncalves}
\address{GR$^2$IF. Université de Rouen-Normandie. Technopôle du Madrillet
BP 12 Avenue de l'Université
F-76801 Saint-Étienne-du-Rouvray Cedex}
\email{ulysse.goncalves@etu.univ-rouen.fr}
\author{Jean-Gabriel Luque}
\address{GR$^2$IF. Université de Rouen-Normandie. Technopôle du Madrillet
BP 12 Avenue de l'Université
F-76801 Saint-Étienne-du-Rouvray Cedex}
\email{jean-gabriel.luque@univ-rouen.fr}
\date{}

\begin{document}
\begin{abstract}
 Motivated by applications to the fractional quantum Hall effect and, in particular, to the Bernevig-Haldane conjectures, we investigates the behavior of Macdonald polynomials under specializations of the form $q^{a}t^{b}=1$.
Our main focus is to explain, in a simple and purely combinatorial way, why certain nonsymmetric Macdonald polynomials indexed by staircase vectors with steps of height $a$ and width $b$ remain regular at the specialization $q^{a}t^{\,b+1}=1$, despite the presence of potential poles in their rational coefficients. To this end, we introduce a set of combinatorial tools that track how poles are created or cancelled along paths in the Yang-Baxter graph. 
By carefully constructing paths from the zero vector to the staircase  and analyzing the resulting denominators, we show that the absence of certain poles follows from intrinsic symmetries and cancellations encoded in the Yang-Baxter graph. \end{abstract}
\maketitle
\section{Introduction}
{\it ``Macdonald polynomials are too difficult!''} This is a remark one often hears from computer scientists when they are asked to automate the computation of these polynomials. They are not entirely wrong. Indeed, as they appear in much of the literature, Macdonald polynomials\cite{Macdonald1988,Macdonald1995,Macdonald2000} suffer from two main difficulties. The first one is conceptual: their motivation stems from deep and notoriously difficult areas of mathematics, such as random matrix theory, generalization of the Dyson’s conjecture, double affine Hecke algebras (DAHA), root systems and Coxeter groups etc. The second is structural: these polynomials are enormous multivariate polynomials whose coefficients, rational functions in the two parameters 
$q$ and 
$t$, at first glance appear almost random and devoid of any discernible structure. For someone who has not studied the subject extensively, learning these objects may legitimately appear inaccessible, and their manipulation genuinely arduous. Nevertheless, with a bit of perspective, these polynomials can be handled in a relatively simple way using fairly elementary algorithmic and combinatorial tools. For instance, the symmetric Macdonald polynomials can be defined via a Gram-Schmidt orthogonalization process, while their nonsymmetric counterparts are characterized by simple recurrence relations. The various applications of these polynomials, particularly, though not exclusively, in physics, call for the development of effective knowledge transfer strategies, so that these objects can be studied in a genuinely interdisciplinary way, spanning computer science, mathematics, and physics. The combinatorial approach pioneered by Alain Lascoux\cite{Lascoux2001,MacdoforDummies,Lascoux2013}, based on the study of the Yang–Baxter graph, provides a way to simplify the understanding of these objects. Our paper fits into this framework.

In order to describe the wave functions associated with states of the fractional quantum Hall effect\footnote{ The knowledge of these wave functions is of particular importance, since certain fault-tolerant models of quantum computation~\cite{Kitaev1997} are based on quasiparticles, called anyons, which arise (at least theoretically) precisely in certain regimes of the fractional quantum Hall effect. This is the setting of topological quantum computation~\cite{FKLW2003}, whose connections with these wave functions and their combinatorics, especially the Read--Rezayi states, have been developed in the comprehensive work~\cite{NSSFD2008}.}, B.~A.~Bernevig and F.~D.~M.~Haldane~\cite{BH2008_1,BH2008_2} proposed candidates expressed in terms of Jack polynomials, with the parameter $\alpha$ specialized to a rational value.
These candidates were selected according to clustering properties, which translate algebraically into vanishing conditions (from a physical point of view, this means that the wave function prevents the particles from clustering “too” strongly). The vanishing conditions translate into factors under the form $(x_i-x_j)^a$. This led to the formulation of three conjectures~\cite{BH2008_2}, one of which was proved in~\cite{CDL2018}. The route that led to the resolution of this conjecture required the use of more general classes of polynomials. First, symmetric Macdonald polynomials, of which Jack polynomials arise as a degeneration, were used, since the two-parameter deformation $(q,t)$ allows one to effectively distinguish the factors in the expressions $(x_i-x_j)^a$ and to reformulate the conjecture in a more general framework that does not necessarily degenerate to Jack polynomials. Subsequently, nonsymmetric Macdonald polynomials made it possible to exploit the combinatorics of the Yang-Baxter graph, as well as the combinatorial structures arising from the representation theory of the double affine Hecke algebra.

This paper focuses on the following subproblem. Since the coefficients of Macdonald polynomials are rational functions of the parameters $q$ and $t$, they may have poles. When one specializes these parameters by imposing relations of the form 
$q^{a}t^{b}=1$, certain Macdonald polynomials become ill-defined. Why, then, are the Macdonald polynomials that appear in the conjectures of Bernevig and Haldane well defined at their corresponding specializations? And, in pratical terms, is there a simple way to prove this? More precisely, we ask why the nonsymmetric Macdonald polynomials indexed by staircase shapes with steps of height $a$ and width $b$ have no poles at the specialization $q^{a}t^{\,b+1}=1$.
With this objective in mind, we develop a collection of combinatorial tools, based on the Yang-Baxter graph, that allow us to understand how poles arise in Macdonald polynomials.

The paper is organized as follows. In Section~\ref{sec-background}, we recall the main definitions and properties of nonsymmetric Macdonald polynomials. In particular, we review their role in the representation theory of double affine Hecke algebras, as well as the recurrence relations that allow one to compute them from the Yang--Baxter graph. In Section~\ref{sec-preminary}, we introduce several notions and tools that will be studied in the following sections. In particular, in Section~\ref{subsec-denom} we study the maps that associate to each path in the Yang-Baxter graph a polynomial related to the denominators of the Macdonald polynomials at the initial and terminal vertices of the path. In Subsection~\ref{subsec-spec}, we make precise the notion of specialization at $q^{a}t^{b}=1$, and we clarify what this means when $a$ and $b$ are not coprime. This is precisely the property used in Section~\ref{sec-jump} to show that only a single pole is produced when passing from an index of the form \([\ldots,a^{k},b,\ldots]\) to \([\ldots,b,a^{k},\ldots]\) with \(b>a\). This result is what we call the Jump Lemma. The dual Jump Lemma, which concerns parts of the form \([\ldots,a,b^{k},\ldots]\), is also described. By iterating these results, we explain in Section~\ref{sec-preminary} which poles arise when following a path
$
[\ldots,a^{k},b^{\ell},\ldots]\;\longrightarrow\;[\ldots,b^{\ell},a^{k},\ldots]$.
From this point, all the necessary combinatorial tools have been introduced, and in Section~\ref{sec-staircase} we illustrate them in order to understand why certain poles do not appear in staircase Macdonald polynomials. We show that this property can be proved in a purely combinatorial way by carefully constructing a path from $[0^{n}]$ to the staircase vector of interest and tracking the poles using the jump results established in the previous sections.

\section{Background\label{sec-background}}

\subsection{Hecke algebra}
The Hecke algebras are obtained by deforming Coxeter groups with a single parameter; this deformation was introduced in order to better reflect and study the geometry and combinatorics of the Bruhat order \cite{KazLus1979,Iwahori1965}.\\
This algebra can be defined formally by generators and relations.  
In the case relevant to our setting, the Hecke algebra $\mathcal{H}_N(t)$ of type~$A_{n-1}$ is the unital associative
$\mathbb{C}(t)$-algebra generated by formal elements $T_1,\dots,T_{n-1}$, called \emph{intertwiners}, that are subject to the following relations:
\begin{itemize}
    \item Quadratic relations:
    \begin{equation}\label{eq-quad}
        (T_i - t)(T_i + 1) = 0, \qquad 1 \le i \le N-1,
    \end{equation}
    \item Braid relations:
    \begin{equation}\label{eq-braid}
        T_i T_{i+1} T_i = T_{i+1} T_i T_{i+1},
        \qquad 1 \le i \le N-2,
    \end{equation}
    \item Commutation relations:
    \begin{equation}\label{eq-com}
        T_i T_j = T_j T_i,
        \qquad |i-j| > 1.
    \end{equation}
\end{itemize}

These relations mirror those of the symmetric group, except that the involution
relation $s_i^2 = 1$ is replaced by its one-parameter deformation.

As in the case of the symmetric group, it is useful to realize this algebra
as an algebra of operators acting on the polynomial algebra
$\mathbb{C}(t)[\mathbb X]$, $\mathbb X=\{x_1,\dots,x_N\}$ denoting the set of (formal) variables. This can be done by assimilating it to the Demazure-Lusztig operator,
\begin{equation}\label{eq-real}
T_i = (t-1)\pi_i + s_i
     = \partial_i (t \vx_{i+1} - \vx_i) + t,
\end{equation}
where $s_i$ denotes the transposition operator exchanging the variables
$x_i$ and $x_{i+1}$, $\vx_i$ is the operator which multiplies by the variable $x_i$, $\partial_i = (1 - s_i)\frac1{\vx_i - \vx_{i+1}}$ is the
divided-difference operator\footnote{Notice that the action of $\partial_i$ on polynomials is well-defined because, for any polynomial $P$, $P(1-s_i)$ is divisible by $x_i-x_{i+1}$.}, and $\pi_i = \vx_i\partial_i\,$ is the isobaric
divided-difference operator. \\\emph{Notice that our operators act on their left.} For instance, we have $x_i^2\partial_i=x_i+x_{i+1}$. The introduction of the notation $\vx_i$ for multiplication operators may
seem artificial, but it helps avoid ambiguities in expressions (for instance,
between $\vx_i \partial_i$ and $x_i \partial_i$).

The right most expression in the double equality \refeq{eq-real} is particularly interesting because it implies that the kernel of the operator $T_i-t$ is the space of the polynomials that are symmetrical in the variables $\{x_i,x_{i+1}\}$.
\begin{claim}\label{claim-kerT_i-t}
Let $P\in\mathbb C(t)[\mathbb X]$. We have $Ps_i=P$ if and only if $P(T_i-t)=0$.
\end{claim}
This simple remark will be of a great interest for our purpose.\\
The Double Affine Hecke Algebra (DAHA) was introduced by Cherednick \cite{Cherednik1995}, in the aim to investigate conjecture on Knizhnik-Zamolodchikov equation and Macdonald polynomials. It is a $\mathbb C(q,t)$-algebra that extends  the previous construction by adding a new formal parameter $q$ and  two sets of generators $\vx_1^{\pm},\ldots, \vx_N^{\pm}$ and $\vy_1^{\pm},\ldots, \vy_N^{\pm}$ submitted to the additional relations:
\begin{itemize}
\item Inversions:
\begin{equation}\label{eq-inversion}
\vx_i^+\vx_i^{-}=Id,\qquad\vy_i^+\vy_i^{-}=Id.
\end{equation}
So in the rest of the paper, we denote $\vx_i=\vx_i^+$ and
$\vy_i=\vy_i^+$
    \item Commutations:
    \begin{equation}\label{eq-homcom}
    \vx_i\vx_j=\vx_j\vx_i\mbox{ and }
    \vy_i\vy_j=\vy_j\vy_i\mbox{ for any }1\leq i,j\leq N,
    \end{equation}
    and
    \begin{equation}\label{eq-hetcomm}
    \vx_i\vy_j=\vy_j\vx_i,\qquad T_i\vx_j=\vx_jT_i,\mbox{ and }T_i\vy_j=\vy_jT_i\mbox{ when }|i-j|>1.
    \end{equation}
    \item Bernstein–Lusztig relations:
    \begin{equation}\label{eq-BLR}
    T_i\vx_iT_i=\vx_{i+1}\mbox{ for any } 1\leq i\leq N-1.
    \end{equation}
    \item Dual-Bernstein–Lusztig relations:
    \begin{equation}\label{eq-dualBLR}
    T_i\vy_{i+1}T_i=\vy_{i}\mbox{ for any } 1\leq i\leq N-1.
    \end{equation}
    \item Cross relations:
    \begin{equation}\label{eq-cross}
    \vy_i\vx_1\vx_2\cdots\vx_N=q\vx_1\vx_2\cdots\vx_N\vy_i\mbox{ for any }1\leq i\leq N,
    \end{equation}
    and
    \begin{equation}\label{eq-dcross}
    \vx_i\vy_1\vy_2\cdots\vy_N=q\vy_1\vy_2\cdots\vy_N\vx_i\mbox{ for any }1\leq i\leq N.
    \end{equation}
\end{itemize}
Cherednik’s key idea is that the interactions between the two commutative subalgebras generated respectively by the $X_i$ and the $Y_i$ are intimately connected with the structure of Macdonald polynomials.
\subsection{Cherednik's representation and non symmetric Macdonald polynomials}
\ \\
 Cherednik’s representation realizes the DAHA as a faithful algebra of operators acting on the Laurent polynomial ring $\mathbb{C}(q,t)[\mathbb{X}^{\pm1}]$. 
In this representation, the intertwiners act as the Demazure--Lusztig operators, exactly as in the classical representation of the Hecke algebra, 
the elements $\vx_i$ act by multiplication by the variables $x_i$, 
and the generators $\vy_i$ are realized as the Cherednik--Dunkl operators
\begin{equation}\label{eq-Cherednik-Dunkl}
\vy_i=T_i\cdots T_{N-1}\tau T_1^{-1}\cdots T_{i-1}^{-1},
\end{equation}
where $\tau$ is defined by
\begin{equation}\label{eq-tau}
(P \tau)(x_1, x_2, \dots, x_N)
= P\!\left(\frac{x_N}{q}, x_1, x_2, \dots, x_{n-1}\right),
\end{equation}
for any Laurent polynomial $P$.\\
Notice that one has the following additional commutation rules
\begin{eqnarray}
T_i\tau=\tau T_{i+1}\mbox{ for }1\leq i<N-1
\\
\vx_i\tau=\tau \vx_{i-1}\mbox{ for }i>1\mbox{ and }\vx_1\tau=\frac1q\tau\vx_N,\label{eq-commXtau}\\
\vy_i\tau=\tau \vy_{i+1}\mbox{ for }1\leq i<N\mbox{ and }\vy_N\tau=q\tau\vy_1\label{eq-commYtau}.
\end{eqnarray}
Since the Cherednik--Dunkl operators pairwise commute, they can be diagonalized simultaneously.  
Their joint spectrum decomposes the space $\mathbb{C}(q,t)[\mathbb{X}]$ as a direct sum of one-dimensional spectral subspaces.
Each of these subspaces is generated by a Macdonald polynomial.  
More precisely, to every vector\footnote{a vector is wrote as a word of length $N$ on the alphabet $\mathbb N$.} $v = v_1\cdot v_2\cdots\cdot v_N$  one associates a Macdonald polynomial $M_v$, defined up to a nonzero scalar, as the unique polynomial with $\triangleright$-leading monomial $x^v = x_1^{v_1}\cdots x_N^{v_N}$ and satisfying
\begin{equation}\label{eq-defMacdo}
M_v\,\vy_i \;=\; q^{v_i}t^{*}\, M_v,\qquad 1\le i\le N,
\end{equation}
 the order $\underline{\triangleright}$ being defined by
\begin{equation}\label{eq-triangleright}
u\;\underline{\triangleright}\;v\mbox{ if and only if }u^+\geq v^+ \mbox{ or }(u^+=v^+\mbox{ and } u\geq v),
\end{equation}
where $w^+$ denotes the unique weakly decreasing vector obtained by permuting the element of $w$, and $\geq$ is the dominant order given by
\begin{equation}\label{eq-dominantorder}
u\geq v\mbox{ if and only if for every } 1\leq k\leq N,\; \sum_{i=1}^ku_i\geq \sum_{i=1}^kv_i.
\end{equation}
Notice that, for the purposes of the statement, we do not need to specify the exact exponent of $t$ in advance; this is the meaning of the notation $t^*$. Indeed, the collection of vectors of the form 
$[\,q^{v_1}t^{*},\ldots,q^{v_N}t^{*}\,]$ 
provides enough distinct spectral values to decompose the space as a direct sum of one-dimensional eigenspaces. In other words, there exists a bijection between $\mathbb N^N$ and the set 
\begin{equation}\label{eq-defspect}
\setspecY=\left\{\spectreY{v}:=\left[\frac{M_v\,\vy_1}{M_v},\ldots,
\frac{M_v\,\vy_N}{M_v}\right]\mid v\in\mathbb N^N\right\}.
\end{equation}
The precise eigenvalue of $M_v$ for the operator $\vy_i$ is
\begin{equation}\label{eq-valspect}
\spectreY{v}[i] = q^{v_i}t^{\#\{j\mid v_j < v_i\mbox{ or }(v_j = v_i\mbox{ and } j \geq i)\}-i}
\end{equation}
For instance, we have 
\begin{equation}\label{eq-specex}
\spectreY{[1,0,2,2,0,1]}=[qt^3, 1, q^2t^3, q^2t, t^{-4}, qt^{-3}].
\end{equation}
In practice, one prefers to work with the operators $\vhy_i = t^{\,i-1}\vy_i$, since their $t$–powers are all distinct and lie within the interval $\{0,\ldots,N-1\}$. More precisely, setting 
\begin{equation}\label{eq_defhspec}
\setspec=\left\{\spectre{v}:=\left[\frac{M_v\,\vhy_1}{M_v},\ldots,
\frac{M_v\,\vhy_N}{M_v}\right]\mid v\in\mathbb N\right\},
\end{equation}
one has
\begin{equation}\label{eq-valhspec}
\spectre{v}[i]=q^{v_i}t^{\std{v}_i-1},
\end{equation}
where $\std v$ denotes the standarized of $v$ that is the unique permutation $\sigma$ of $\mathfrak S_N$ such that 
\begin{equation}\label{eq-defstd}
\sigma_i > \sigma_j\mbox{ if and only if } v_i>v_j\mbox{ or } (v_i = v_j\mbox{ and }i<j). 
\end{equation}
For instance, we have
\begin{equation}\label{eq-exhspec}
\spectre{[1,0,2,2,0,1]}=[qt^3,t,q^2t^5,q^2t^4,1,qt^2]\mbox{ because }\std v=426513.
\end{equation}
In this way, the set of spectra becomes easier to describe. With this notation, the  Dual-Bernstein-Lusztig relations read
\begin{equation}\label{eq-dBLhat}
T_i\vhy_{i+1}T_i=t\vhy_i\mbox{ for }0\leq i\leq N-1.
\end{equation}
We will later make use of the following property, whose proof is immediate from the definition of~$\setspec$.

\begin{claim}\label{claim-spectre}
Let $P$ be a nonzero simultaneous eigenfunctions of the operators $\vhy_1,\,\ldots,\,\vhy_N$ such that for some $1\leq i\leq N-1$, $\frac{P\vhy_i}{P\vhy_{i+1}}=t^a$. Then we have  $a=1$.
\end{claim}
\subsection{Yang-Baxter graph}
From a computational point of view, Macdonald polynomials can be computed by applying the following induction rules:
\begin{eqnarray}
    M_{0^N}=1,\label{eq-MacdoInd1}\\
    M_{v\cdot s_i}=M_v\left(T_i+\frac{1-t}{1-\frac{\spectre v[i+1]}{\spectre v[i]}}\right)&\mbox{ if }v_{i+1} > v_i\label{eq-MacdoInd2}\\
    M_{v\Phi}=M_v\tau x_N&\mbox{ with }v\Phi = v_2 \cdot v_3 \cdots v_n \cdot (v_1+1).\label{eq-MacdoInd3}
\end{eqnarray}
In the sequel we denote $\affineop=\tau x_N$ and $\Yang \alpha i=T_i+\frac{1-t}{1-\alpha}$.
These inductive rules are established by proving that the right-hand sides are simultaneous eigenfunctions of the Dunkl--Cherednik operators.  
The precise eigenvalues, calculated from the relations among the generators, determine which Macdonald polynomial appears on the left-hand side.

Rules \refeq{eq-MacdoInd1}, \refeq{eq-MacdoInd2} and \refeq{eq-MacdoInd3} make it possible to apply an algorithm for computing Macdonald polynomials based on simple combinatorial considerations. 
Indeed, it suffices to construct a sequence of actions involving either transpositions $s_i$ applied to increasing parts or the operators~$\Phi$ 
and then apply the corresponding operator (resp. \refeq{eq-MacdoInd2} or \refeq{eq-MacdoInd3}) from the polynomial $1$. More precisely, we consider a directed graph whose vertices are the elements of $\mathbb{N}^N$, 
with edges of the form $v \xrightarrow[]{s_i} v s_i$ whenever $v_i+1 > v_{i+1}$, and 
$v \xrightarrow[]{\Phi} v\Phi$.  Classically, a path is a succession of arrow $v_1\xrightarrow[]{a_1}v_2\xrightarrow[]{a_2}\cdots\xrightarrow[]{a_{n-1}}v_n$ with $a_{i}\in\{s_i\mid 1\leq i\leq N-1\}\cup\{\Phi\}$ and we denote it by $v_1\xrightarrow{a_1a_2\cdots a_{n-1}}v_n$.
By following the edges of a path 
$
0^N \xrightarrow[]{\text{path}} v
$
and applying successively the operators that label the edges, we obtain an algorithm that computes any vector starting from the zero vector.  
This algorithm is straightforwardly confluent: the result depends only on the terminal vertex $v$, not on the chosen path leading to it. By replacing the affine edge $\Phi$ with $\affineopvec$, defined by $[a_1,\ldots,a_N]\affineopvec=[a_2,\ldots,a_N,qa_1]$, and starting from the initial state 
$
[t^{N-1},\, t^{N-2},\, \ldots,\, 1]$,
one computes the corresponding spectral vector. For instance, consider the vector $v=102$, it is obtained from $000$ following the path
\begin{equation}\label{eq-ex102}
000\xrightarrow[]{\Phi}001\xrightarrow[]{s_2}010
\xrightarrow[]{\Phi}101\xrightarrow[]{\Phi}012\xrightarrow[]{s_1}102
\end{equation}
its spectral vector is obtained by the following succession of operations
\begin{equation}
[t^2,t,1]\xrightarrow[]{\affineopvec}[t,1,qt^2]\xrightarrow[]{s_2}[t,qt^2,1]
\xrightarrow[]{\affineopvec}[qt^2,1,qt]\xrightarrow[]{\affineopvec}
[1,qt,q^2t^2]\xrightarrow[]{s_1}[qt,1,q^2t^2].
\end{equation}
From equalities \refeq{eq-MacdoInd1}, \refeq{eq-MacdoInd2}, \refeq{eq-MacdoInd3}, we can compute any polynomial $M_v$ from $M_{0^N}=1$ by following any path from $0^N$ to $v$ and by applying the operator $\affineop$ whenever we traverse an edge labeled by $\Phi$, and the operator 
$\Yang{\frac{\spectre v[i+1]}{\spectre v[i]}}i$ whenever we traverse an edge labeled by $s_i$.
For instance, we compute $M_{102}$ by following the path described in \refeq{eq-ex102} and by applying the succession of operators below:
\begin{equation}
1\xrightarrow[]{\affineop}M_{001}\xrightarrow[]{\Yang{qt^2}2}M_{010}
\xrightarrow[]{\affineop}M_{101}\xrightarrow[]{\affineop}M_{012}\xrightarrow[]{\Yang{qt}1}M_{102}
\end{equation}
Hence,
\begin{equation}
M_{102}=1\cdot\affineop\Yang{qt^2}2\affineop^2\Yang{qt}1 =
{\frac {{\it x_1}\,{{\it x_3}}^{2}}{q}}+{\frac { \left( t-1 \right) {\it x_3}\,{\it x_1}\,{\it x_2}}{tq-1}}+{\frac { \left( t-1 \right) {\it x_2}\,{{\it x_3}}^{2}}{q \left( tq-1 \right) }}
\end{equation}
Algorithmically, it is therefore enough to follow the computation on the three graphs simultaneously while traversing the same path.  
The confluence property implies that it suffices to exhibit \emph{any} path to perform the computation.  
For instance, in Figure~\ref{f-M102}, we see four possible paths that allow us to compute $M_{102}$.

\begin{figure}[ht]
\begin{center}
\tiny
\begin{tikzpicture}[scale=0.75]
\GraphInit[vstyle=Shade]
    \tikzstyle{VertexStyle}=[shape = rectangle,
draw
]
\SetUpEdge[lw = 1.5pt,
color = orange,
 labelcolor = gray!30,
 style={post}
]
\tikzset{LabelStyle/.style = {draw,
                                     fill = white,
                                     text = black}}
\tikzset{EdgeStyle/.style={post}}
\Vertex[x=0, y=8,
 L={$000$}]{x}
\Vertex[x=0, y=6,
 L={$001$}]{y}
 \Vertex[x=-2, y=4,
 L={$010$}]{z1}
 \Vertex[x=2, y=4,
 L={$011$}]{z2}
 \Vertex[x=0, y=2,
 L={$101$}]{t}
 \Vertex[x=-2, y=0,
 L={$110$}]{s1}
 \Vertex[x=2, y=0,
 L={$012$}]{s2}
 \Vertex[x=0, y=-2,
 L={$102$}]{u}
\Edge[label={$\Phi$}](x)(y)
\Edge[label={$s_2$}](y)(z1)
\Edge[label={$\Phi$}](y)(z2)
\Edge[label={$\Phi$}](z1)(t)
\Edge[label={$s_1$}](z2)(t)
\Edge[label={$s_2$}](t)(s1)
\Edge[label={$\Phi$}](t)(s2)
\Edge[label={$\Phi$}](s1)(u)
\Edge[label={$s_1$}](s2)(u)

\SetUpEdge[lw = 1.5pt,
color = green,
 labelcolor = gray!30,
 style={post}
]
\tikzset{LabelStyle/.style = {draw,
                                     fill = white,
                                     text = black}}
\Vertex[x=6, y=8,
 L={$[t^2,t,1]$}]{fx}
\Vertex[x=6, y=6,
 L={$[t,1,qt^2]$}]{fy}
 \Vertex[x=4, y=4,
 L={$[t,qt^2,1]$}]{fz1}
 \Vertex[x=8, y=4,
 L={$[1,qt^2,qt]$}]{fz2}
 \Vertex[x=6, y=2,
 L={$[qt^2,1,qt]$}]{ft}
 \Vertex[x=4, y=0,
 L={$[qt^2,qt,1]$}]{fs1}
 \Vertex[x=8, y=0,
 L={$[1,qt,q^2t^2]$}]{fs2}
 \Vertex[x=6, y=-2,
 L={$[qt,1,q^2t^2]$}]{fu}
\Edge[label={$\affineopvec$}](fx)(fy)
\Edge[label={$s_2$}](fy)(fz1)
\Edge[label={$\affineopvec$}](fy)(fz2)
\Edge[label={$\affineopvec$}](fz1)(ft)
\Edge[label={$s_1$}](fz2)(ft)
\Edge[label={$s_2$}](ft)(fs1)
\Edge[label={$\affineopvec$}](ft)(fs2)
\Edge[label={$\affineopvec$}](fs1)(fu)
\Edge[label={$s_1$}](fs2)(fu)

\SetUpEdge[lw = 1.5pt,
color = blue,
 labelcolor = gray!30,
 style={post}
]
\tikzset{LabelStyle/.style = {draw,
                                     fill = white,
                                     text = black}}

\Vertex[x=12, y=8,
 L={$M_{000}$}]{ffx}
\Vertex[x=12, y=6,
 L={$M_{001}$}]{ffy}
 \Vertex[x=10, y=4,
 L={$M_{010}$}]{ffz1}
 \Vertex[x=14, y=4,
 L={$M_{011}$}]{ffz2}
 \Vertex[x=12, y=2,
 L={$M_{101}$}]{fft}
 \Vertex[x=10, y=0,
 L={$M_{110}$}]{ffs1}
 \Vertex[x=14, y=0,
 L={$M_{012}$}]{ffs2}
 \Vertex[x=12, y=-2,
 L={$M_{102}$}]{ffu}
\Edge[label={$\affineop$}](ffx)(ffy)
\Edge[label={$\Yang{qt^2}2$}](ffy)(ffz1)
\Edge[label={$\affineop$}](ffy)(ffz2)
\Edge[label={$\affineop$}](ffz1)(fft)
\Edge[label={$\Yang{qt^2}2$}](ffz2)(fft)
\Edge[label={$\Yang{qt}2$}](fft)(ffs1)
\Edge[label={$\affineop$}](fft)(ffs2)
\Edge[label={$\affineop$}](ffs1)(ffu)
\Edge[label={$\Yang{qt}1$}](ffs2)(ffu)

\end{tikzpicture}
\end{center}
\caption{Computation of $M_{102}$. \label{f-M102}}
\end{figure}
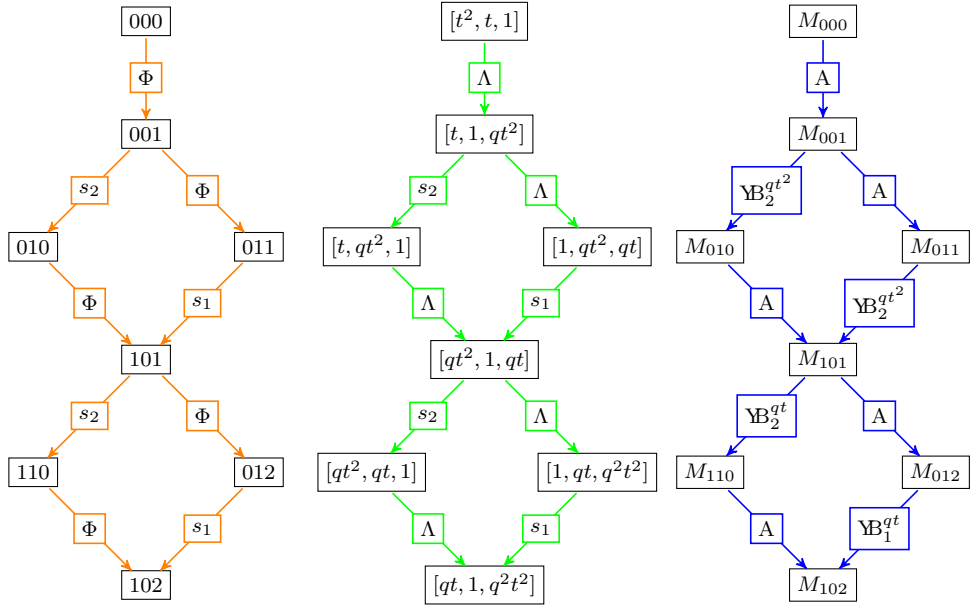
The underlying structure is known as the Yang--Baxter graph.\\ 
It is noteworthy that the order $\underline{\triangleright}$ is closely connected to the 
logical structure of the Yang--Baxter graph. Indeed, we have
\begin{equation}\label{eq-compTPhitriangle}
    vs_i\;\triangleright\; v\mbox{ if }v_i<v_{i+1},\mbox{
and }
    v\Phi\;\triangleright\; v,
\end{equation}
and this implies
\begin{equation}\label{eq-compYBtriangle}
u\;\underline\triangleleft\; v\mbox{ when }u\xrightarrow[]{\mathtt{path}}v.
\end{equation}
Furthermore, we have
\begin{equation}\label{eq-morphTi}
(v_i<v_{i+1}\mbox{ and }u\;\underline\triangleleft\; v)\mbox{ implies }us_i\;\underline\triangleleft\;vs_i,
\end{equation}
and
\begin{equation}\label{eq-morphPhi}
u\;\underline\triangleleft\; v\mbox{ implies }u\Phi\;\underline\triangleleft\;v\Phi,
\end{equation}
Hence, the property of the $\triangleright$-leading term of the Macdonald polynomials follows from the compatibility 
between the operators $T_i$, $\affineop$, with the order $\underline{\triangleright}$.

The Yang--Baxter graph turns out to be extremely useful for computing or proving various properties inductively.  
Let us illustrate this technique by sketching the proof that the leading coefficient of $M_v$ is  
$q^{-\frac12 \sum_{i=1}^{N} v_i (v_i - 1)}$.
The initialization of the induction is straightforward, since $M_{0^N}=1$.  
It remains to show that the property is invariant along the graph.  
That is, if $M_v$ satisfies the property and there is an edge 
$v \xrightarrow[]{*} v'$ in the Yang--Baxter graph, then $v'$ satisfies it as well.
There are two cases to consider.  
The first one corresponds to edges of the form 
$v \xrightarrow[]{s_i} v'=v \cdot s_i$ with $v_i < v_{i+1}$. When applying $\Yang{*}i$ the leading monomial becomes $x^{v'}$ and the leading coefficient does not change. Using the induction hypothesis, it is equals to $q^{-\frac12\sum_iv_i(v_i-1)}=q^{-\frac12\sum_iv'_i(v'_i-1)}$. The second one corresponds to edges of the form $v\xrightarrow[]{\Phi}v'=v\cdot\Phi$. When applying $\affineop$, the leading monomial becomes $x^{v'}$ and the leading coefficients is multiplicated by $q^{-v_1}=q^{1-(v\Phi)_N}$. So, from the induction hypothesis the leading coefficient of $M_{v}'$ is
\begin{equation}q^{-\frac12\sum_iv_i(v_i-1)-v_1}=q^{-\frac12\sum_{i\neq N} v'_i(v'_i-1)-\frac12(v'_N-1)(v'_N-2)-v'_N+1}=q^{-\frac12\sum_iv'_i(v'_i-1)}\end{equation}
Hence,
\begin{equation}
    M_v=q^{-\frac12\sum_iv_i(v_i-1)}x^v+\sum_{u\triangleleft v}\alpha_ux^u.
\end{equation}
\section{Preliminary results\label{sec-preminary}}
\subsection{On the Denominators of Macdonald Polynomials\label{subsec-denom}}

We are interested in the poles of Macdonald polynomials. In other words, we seek methods to compute the denominators of these polynomials; that is, we will write a Macdonald polynomial in the form of a reduced fraction
$
M_v = \frac{P_v}{\Den v}$,
where $P_v$ is a polynomial in the variables $\{x_1, \dots, x_N, q, t\}$, and $\Den v$ is a polynomial in the variables $q$ and $t$.
For instance, we have
\begin{equation}
    M_{102}={\frac {{\it x_3}\, \left( qt{\it x_1}\,{\it x_2}+qt{\it x_1}\,{\it x_3}-q{
\it x_1}\,{\it x_2}+t{\it x_2}\,{\it x_3}-{\it x_1}\,{\it x_3}-{\it x_2}\,{
\it x_3} \right) }{q \left( tq-1 \right) }}.
\end{equation}
So,
$
D_{102}={q \left( tq-1 \right) }$.\\
If $P$ et $Q$ are two polynomials, we denote by $\num{\frac PQ}$ (resp. $\den{\frac PQ}$) the numerator (resp. the denominator) of the reduced fraction of $\frac PQ$.
Hence, $\Den v=\den{M_v}$. Our aim is to determine whether a Macdonald polynomial remains well-defined after specializing the parameters $q$ and $t$. To achieve this goal, we focus on polynomials having a factor of the form $\num{\frac{\Den v}{\Den u}}$.\\ Indeed, suppose that we have an algorithm $\algo$ which takes as input a path $u \xrightarrow[]{\text{path}} v$ in the Yang-Baxter graph and returns as output a polynomial which is divisible by $\num{\frac{\Den v}{\Den u}}$. Then, if the polynomial $M_u$ remains well-defined after a certain specialization of the parameters $q$ and $t$, and if $\algo(\mathrm{path})$ does not vanish under the same specialization, the polynomial $M_v$ is also well-defined after this specialization.
We do not assume any specific property of the algorithm $\algo$; in particular, we do not require it to be confluent. That is, if
$
u \xrightarrow[]{\text{path}_1} v$ and $
u \xrightarrow[]{\text{path}_2} v$
are two paths from $u$ to $v$ in the Yang–Baxter graph, we allow for the possibility that
$\algo(\mathrm{path}_1) \neq \algo(\mathrm{path}_2)$.
Nevertheless, whatever the algorithm \(\algo\) may be, it allows us to deduce properties of the polynomials \(\num{\frac{\Den v}{\Den u}}\) following the two simple rules below:
\begin{itemize}
    \item{\bf Conjonction.} If we have $u \xrightarrow[]{\text{path}_1} v \xrightarrow[]{\text{path}_2} w$  then $\num{\frac {\Den w}{\Den u}}$ divides the polynomial $\algo(\mathrm{path}_1)\algo(\mathrm{path}_2)$.
    \item{\bf Disjonction.} If we have $u \xrightarrow[]{\text{path}_1} v$ and $u \xrightarrow[]{\text{path}_2} v$ then $\num{\frac {\Den v}{\Den u}}$ divides the polynomial $\gcd(\algo(\mathrm{path}_1),\algo(\mathrm{path}_2))$.
\end{itemize}
The simplest of these algorithms is the one we shall call \(\triv\), which consists in making appear all denominators induced by the Yang–Baxter graph. It is defined inductively by
\begin{equation}
\triv(\empty) = 1,
\end{equation}
\begin{equation}
\triv(u \xrightarrow[]{\mathrm{path}} v
\xrightarrow[]{\affineop} w) = 
q^{\frac12\sum_{i}(v_i-u_i)(v_i-u_i-1)}\triv(u \xrightarrow[]{\mathrm{path}} v),
\end{equation}
and
\begin{equation}
\triv(u \xrightarrow[]{\mathrm{path}} v
\xrightarrow[]{\Yang{q^at^b}i} w) = 
(1-q^at^b)\triv(u \xrightarrow[]{\mathrm{path}} v).
\end{equation}
The fact that $\num{\frac{\Den v}{\Den u}}$ follows straigthforwardly from the definition of the operators $f$ and $\Yang{q^at^b}i$.\\
At the opposite end of the spectrum lies the algorithm $\opt$, which returns the value $\num{\frac{\Den v}{\Den u}}$ (obtained by brute force), regardless of the path $u \rightarrow v$.\\
We  equip the set of all the considered algorithms with the partial order $\algo_1 \leq \algo_2$ if and only if, for every path $\mathfrak p$ in the Yang–Baxter graph, $\algo_1(\mathfrak p)$ divides $\algo_2(\mathfrak p)$.  
If we consider only the outputs of the algorithms (independently of the sequence of instructions used to obtain them), that is, if we view the algorithms as maps from the set of paths to $\mathbb{C}[q,t]$, then this partial order forms a lattice whose minimal element is $\opt$ and maximal element is $\triv$.
Indeed, the join $\algo_1 \vee \algo_2$ (resp. the meet $\algo_1 \wedge \algo_2$) of two algorithms $\algo_1$ and $\algo_2$ is defined, for any path $\mathfrak{p}$, by
$
(\algo_1 \vee \algo_2)(\mathfrak{p}) = \lcm\big(\algo_1(\mathfrak{p}),\, \algo_2(\mathfrak{p})\big)
$,
resp.
$
(\algo_1 \wedge \algo_2)(\mathfrak{p}) = \operatorname{gcd}\big(\algo_1(\mathfrak{p}),\, \algo_2(\mathfrak{p})\big).
$
More precisely, any algorithm we consider can be constructed from the algorithm $\triv$ by replacing the image of certain paths for which smaller denominators have been identified.  
The Disjunction Rule then allows us to choose the greatest common divisor $D$ of all denominators corresponding to the paths connecting a given pair of Macdonald polynomials.  
We can therefore replace the image of all these paths by $D$ and obtain a smaller algorithm. In this paper, we illustrate this strategy using simplifications based on symmetry properties.


\subsection{Specializations at \texorpdfstring{$q^a t^b = 1$}{q\^at\^b=1}\label{subsec-spec}}
We focus on specializations that cause Macdonald polynomials to degenerate, as these are closely connected to their poles. 
More precisely, we consider specializations satisfying $q^{a} t^{b} = 1$ for some $a,b \in \mathbb{N}$. The construction is as follows.  
If $M_u$ is a Macdonald polynomial, then
\begin{equation}
M_u = \frac{\num{M_u}}{\Den u},
\end{equation}
where $\num{M_u} \in \mathbb{C}[x_1, \dots, x_N, q, t]$ and $\Den u \in \mathbb{C}[q,t]$. \\
Specializing the Macdonald polynomial $M_u$ means canonically projecting $\num{M_u}$ onto the space\\
$\mathbb{C}[x_1, \dots, x_N, q, t]/_{q^{a}t^{b} - 1}$ (resp. $\Den u$ onto $\mathbb{C}[q,t]/_{q^{a}t^{b} - 1}$) in order to obtain a polynomial\\ $N_u(x_1, \dots, x_N; q, t)$ (resp. $D_u(q,t)$).  
When $D_u(q,t) \neq 0$, we denote
\begin{equation}
M_u(x_1, \dots, x_N; q^{a}t^{b}=1) = \frac{N_u(x_1, \dots, x_N; q, t)}{D_u(q,t)}.
\end{equation}
When the context allows, we shall omit the variables $x_1, \dots, x_N$ from the notation.  
If $D_u(q,t) = 0$, we say that the polynomial $M_u$ \emph{degenerates at} $q^{a}t^{b} = 1$.

From a practical point of view, if one wishes to implement the specialization in a computer algebra system, it may be convenient to proceed as follows.  
If $d = \gcd(a,b)$ and $\omega$ is a primitive \(a\)th root of unity such that $\omega^{\frac ad}$ is a primitive $d$-root, then we set
$q = u^{-\frac{b}{d}}$ and $t = \omega\, u^{\frac{a}{d}}$. 
This choice ensures that $q^{a}t^{b} = 1$, while preventing $q^{a'}t^{b'} = 1$ for any smaller integers $a' < a$ and $b' < b$. For instance, since $\Den{310}=q^3(1-q^2t)(1-qt)(1-q^3t^2)$, the polynomial $M_{310}$ degenerates at $qt=1$ but not at $q^2t^2=1$.\\
The next proposition will be particularly useful in manipulating the denominators $\Den u$.
\begin{prop}\label{prop-specomega}
Let $a,b$ be two positive integers and let
\begin{equation}
P(q,t)=\prod_i (1-q^{a_i}t^{b_i})
\end{equation}
be a bivariate polynomial. The following statements are equivalent:
\begin{enumerate}
\item $1-q^{a}t^{b}$ divides $P(q,t)$;
\item there exists an index $i$ such that $1-q^{a}t^{b}$ divides $1-q^{a_i}t^{b_i}$;
\item $P(\omega\,u^{-b/d},u^{a/d})=0$, where $d=\gcd(a,b)$ and $\omega$ is a primitive $a$-th root of unity such that $\omega^{a/d}$ is a primitive $d$-th root of unity (that is,
$
\omega=\exp\!\left(\frac{2\pi i(1+\alpha d)}{a}\right)$ for some $\alpha\in\mathbb Z$).
\end{enumerate}
\end{prop}
\begin{proof}
The implication $(2)\Rightarrow(1)$ is immediate.

To show $(1)\Rightarrow(3)$, assume that $1-q^{a}t^{b}$ divides $P(q,t)$. For any value of $u$ and any $\omega$ such that $\omega^{a}=1$, we have 
$1-q^{a}t^{b}=0$ when $(q,t)=(\omega u^{-b/d},u^{a/d})$. 
Since $1-q^{a}t^{b}$ is a factor of $P$, this implies
$
P(\omega\,u^{-b/d},u^{a/d})=0,
$
and in particular this holds for any primitive $\omega$ such that $\omega^{a/d}$ is a primitive $d$-th root of unity.

Now, let us show $(3)\Rightarrow(2)$. Write $d=\gcd(a,b)$ and set $a=d a'$, $b=d b'$ with $\gcd(a',b')=1$. Suppose that there exists such a root $\omega$ satisfying
$
P(\omega\,u^{-b/d},u^{a/d})=0.
$
Then
\begin{equation}
\prod_i\Bigl(1-\omega^{a_i}\,u^{\frac{-b a_i + a b_i}{d}}\Bigr)=0.
\end{equation}
This implies that there exists $i$ such that 
$
\frac{-b a_i + a b_i}{d}=0$ and $\omega^{a_i}=1$.
The first equality yields $a b_i = b a_i$. We write $a=d a'$ and $b=d b'$ with $\gcd(a',b')=1$ and, since $\gcd(a',b')=1$, we obtain there exist an integer $s$ such that 
$a_i = s a'$, $b_i = s b'$.
The condition $\omega^{a_i}=1$ is equivalent to $(\omega^{a'})^s=1$. Since $\omega^{a'}=\omega^{a/d}$ is by hypothesis a primitive $d$-th root of unity, we have $(\omega^{a'})^s=1$ if and only if $d\mid s$. Writing $s=d r$, we get $
a_i = r a$ and $b_i = r b$. Hence,
\begin{equation}
1-q^{a_i}t^{b_i} = 1-(q^{a}t^{b})^{r}=(1-q^at^b)\sum_{j=0}^{r-1}(q^at^b)^j,\end{equation}
so that $1-q^{a}t^{b}$ divides $1-q^{a_i}t^{b_i}$.\\
Hence the three statements are equivalent.
\end{proof}
Some non-degenerated Macdonald polynomials exhibit remarkable conjectural factorization properties.  
This is particularly the case for those lying in the kernels of the Dunkl operators after specialization; such polynomials are called \emph{singular}.  
It would be convenient to study these factorization formulas directly from the Yang–Baxter graph.  
Unfortunately, certain specializations may prevent us from following the Yang–Baxter graph consistently.

As an example, consider the polynomial $M_{3210}(qt^2 = 1)$.  
This polynomial is singular and has the following remarkable factorization
\begin{equation}
M_{3210}(qt^2 = 1)
= t^{8} (t x_4 - x_3)(t x_4 - x_2)(t x_3 - x_2)
  (t x_4 - x_1)(t x_3 - x_1)(t x_2 - x_1).
\end{equation}
To construct it, one may start from another polynomial that does not degenerate at $qt^2 = 1$, namely
\begin{equation}\begin{array}{l}
M_{2100}(qt^2 = 1)= \left( x_1-t{\it x_2} \right)\times\\\times  \left( {t}^{2}{{\it x_3}}^{2}+{t}^
{2}{\it x_3}\,{\it x_4}+{t}^{2}{{\it x_4}}^{2}-t{\it x_1}\,{\it x_3}-t{\it 
x_1}\,{\it x_4}-t{\it x_2}\,{\it x_3}-t{\it x_2}\,{\it x_4}+{\it x_1}\,{\it 
x_2} \right) {t}^{2},\end{array}
\end{equation}
and then follow the appropriate transformations in the Yang–Baxter graph. Although it is non-degenerate, every path leading to it passes through a degenerate polynomial.  
Indeed, all the paths that allow us to construct $M_{2100}$ end with one of the following three paths:
\begin{itemize}
    \item $1002\xrightarrow[]{s_3} 1020\xrightarrow[]{s_2} 1200\xrightarrow[]{s_1} 2100$,
    \item $0120\xrightarrow[]{s_1} 1020\xrightarrow[]{s_2} 1200\xrightarrow[]{s_1} M_{2100}$,
    \item or $2010\xrightarrow[]{s_2} M_{2100}$.
\end{itemize}
But $\Den{1002}=q(1-qt)(1-qt^2)$,
$\Den{0120}=q(1-qt^2)(1-q^2t^3)$, and $\Den{2010}=q(1-qt)(1-qt^2)(1-q^2t^2)$. So the polynomials $M_{1002}$, $M_{0120}$, and $M_{2010}$ degenerate at $qt^2=1$.\\
It is therefore particularly interesting to develop tools that make it easy to check whether a polynomial degenerates, and that allow us to \emph{jump} over degenerate polynomials in the Yang-Baxter graph. Subsection~\ref{subsec-Sym} is particularly important. It recalls, providing a proof based on the Yang--Baxter graph, that nonsymmetric Macdonald polynomials whose indices have consecutive equal parts are symmetric in the variables corresponding to these parts. This property will be crucial in what follows, as it explains why certain poles may disappear when one follows a path in the Yang-Baxter graph in order to generate a Macdonald polynomial.

\subsection{The Yang-Baxter graph, equal Consecutive Parts and Symmetry\label{subsec-Sym}}
The inductive rule~\refeq{eq-MacdoInd2} arises from the following more general property showing that the Yang–Baxter graph must be understood within a broader framework.
\begin{lem}\label{lem-nonaffineinduction}
Let $P$ be a polynomial.
\begin{enumerate}
    \item If $P$ is an eigenfunction of $\vhy_i$ with eigenvalue $\alpha$, then for any 
$\beta\in\mathbb{C}(q,t)$ and any index $j$ satisfying $|i-j|>1$, the polynomial 
$P\,\Yang{\beta}{j}$ is likewise an eigenfunction of $\vhy_i$ with the same 
eigenvalue~$\alpha$.
\item If $P$ is a simultaneous eigenfunction of $\vhy_i$ and $\vhy_{i+1}$ with respective eigenvalues $\alpha_i$ and $\alpha_{i+1}$ then the polynomial $P\Yang{\frac{\alpha_{i+1}}{\alpha_i}}i$ is likewise a simultaneous eigenfunction of $\vy_i$ and $\vy_{i+1}$ with respective eigenvalues $\alpha_{i+1}$ and $\alpha_i$.
\end{enumerate}
\end{lem}
\begin{proof}
\begin{enumerate}
    \item It is a direct consequence of the commutation rules \refeq{eq-hetcomm}.
    \item One has
    \begin{equation}
    \begin{array}{rcl}
    P\Yang{\frac{\alpha_{i+1}}{\alpha_i}}i\vhy_i&=& 
    P\left(T_i\vhy_i+\frac{1-t}{1-\frac{\alpha_{i+1}}{\alpha_i}}\vhy_i\right)\\
    &=&P\left(\frac1tT_i^2\vhy_{i+1}T_i+\frac{1-t}{1-\frac{\alpha_{i+1}}{\alpha_i}}\vhy_i\right)\\
    &=&P\left(\frac1t((t-1)T_i+t)\vhy_{i+1}T_i+\frac{1-t}{1-\frac{\alpha_{i+1}}{\alpha_i}}\vhy_i\right)\\
    &=&P\left(\vhy_{i+1}T_i+\frac{\alpha_{i+1}}{\alpha_i}\frac{1-t}{1-\frac{\alpha_{i+1}}{\alpha_i}}\vhy_i\right)\\
    &=&P\left(\alpha_{i+1}T_i+\alpha_{i+1}\frac{1-t}{1-\frac{\alpha_{i+1}}{\alpha_i}}\right)=\alpha_{i+1}P\Yang{\frac{\alpha_{i+1}}{\alpha_i}}i,
    \end{array}
    \end{equation}
    and
    \begin{equation}
        \begin{array}{rcl}
         P\Yang{\frac{\alpha_{i+1}}{\alpha_i}}i\vhy_{i+1}&=& P\left(T_i\vhy_{i+1}+\frac{1-t}{1-\frac{\alpha_{i+1}}{\alpha_i}}\vhy_{i+1}\right)\\
         &=&P\left(t\vhy_iT_i^{-1}+\frac{1-t}{1-\frac{\alpha_{i+1}}{\alpha_i}}\vhy_{i+1}\right)\\
         &=&P\left(\vhy_i(T_i+(1-t))+\frac{1-t}{1-\frac{\alpha_{i+1}}{\alpha_i}}\vhy_{i+1}\right)\\
         &=&P\left(\alpha_iT_i+(1-t)\left(\alpha_i+\frac{\alpha_{i+1}}{1-\frac{\alpha_{i+1}}{\alpha_i}}\right)\right)\\
         &=&\alpha_iP\Yang{\frac{\alpha_{i+1}}{\alpha_i}}i.
        \end{array}
    \end{equation}
%
\end{enumerate}
\end{proof}

This property, which is relatively simple compared with the rest of Cherednik’s theory, is truly fundamental. Indeed, it makes it possible to define other representations of the DAHA, for instance on vectorially graded polynomials. It is also very convenient for establishing, in an elegant manner, certain properties of Macdonald polynomials. To illustrate these remarks, we prove below the symmetry property of Macdonald polynomials with consecutive equal parts. This property is regarded as folklore and can be proved in many different ways. Nevertheless, it is particularly straightforward to establish when viewed through the Yang–Baxter graph. It will also play a central role in our later discussions.
\begin{prop}\label{p-symM}
Let $v\in\mathbb N^N$ such that $v[i]=v[i+1]$ for some $i\in\{1,\dots,N\}$. The polynomial $M_v$ is symmetric in $\{x_i,x_{i+1}\}$. 
\end{prop}
\begin{proof}
From Claim \ref{claim-kerT_i-t}, it suffices to prove that $M_v$ is in the kernel of $T_i-t$. Since $v_i=v_{i+1}$, Claim \ref{claim-spectre} implies $t\spectre v[i+1]=\spectre v[i]$ and then $T_i-t=\Yang{\frac{\spectre v[i+1]}{\spectre v[i]}}i$. Applying Lemma \ref{lem-nonaffineinduction} we obtain that the polynomial $M_v(T_i-1)=M_v\Yang{\frac{\spectre v[i+1]}{\spectre v[i]}}i$ is a simultaneous eigenfunction of the operators $\vhy_i$'s whose spectrum is $\spectre v\cdot s_i$. Hence, we have
$\frac{M_v(T_i-t)\vhy_i}{M_v(T_i-t)\vhy_{i+1}}=t^{-1}$. This implies, by Claim \ref{claim-spectre}, that $M_v(T_i-t)=0$ and so $M_v$ is symmetrical in the variables $\{x_i,x_{i+1}\}$.
\end{proof}


To conclude this subsection, we observe that a polynomial lies in the kernel of $T_i - t$ if and only if it lies in the kernel of $\Yang\alpha i - \frac{1 - t\alpha}{1 - \alpha}$ for $\alpha \notin \{1, t^{-1}\}$. Consequently, if $P$ is a polynomial symmetric in the variables $x_i$ and $x_{i+1}$, then
\begin{equation}\label{eq-YangSym}
P\Yang{\alpha}i=\frac{1-t\alpha}{1-\alpha}P
\end{equation}
This property will be extensively exploited to compute the factors of the denominators of Macdonald polynomials.

\section{Jumping lemmas\label{sec-jump}}
We show that a jump over a block of identical values in the Yang-Baxter graph adds at most one pole to the Macdonald polynomial. Let us illustrate what happens with an example. Consider the polynomial $M_{022230}$. We have $\Den{022230}=q^6(1-q^3t^5)(1-q^2t^4)(1-q^2t^2)(1-qt)$. Acting by $\Yang{qt^3}4$ gives
\begin{equation}\label{eq-M022320}
M_{022320}=M_{022230}\Yang{qt^3}4=
M_{022230}T_4+\frac{1-t}{1-qt^3}M_{022230},
\end{equation}
and
$\Den{022320}=(1-qt^3)\Den{022230}$. The extra denominator $1 - q t^{3}$ arises from the second term on the right-hand side of Equality \refeq{eq-M022320}. To obtain $M_{023220}$, we act by $\Yang{qt^2}3$
\begin{equation}
M_{023220}=M_{022320}\Yang{qt^2}3,
\end{equation}
and we observe that $\Den{023220}=\frac{1-qt^2}{1-qt^3}\Den{022320}$. As previously, the presence of the factor $1-qt^2$ straightforwardly comes from the expression of $\Yang{qt^2}3$. The fact that the factor $1-q t^{3}$ cancels out follows from the equality
\begin{equation}\label{eq-M023220}
M_{023200}=M_{022230}T_4\Yang{qt^2}3 + \frac{1-t}{1-qt^3}M_{022230}\Yang{qt^2}3.
\end{equation}
Indeed, since by Proposition \ref{p-symM} the polynomial $M_{022230}$ is symmetric in $x_3$ and $x_4$, Equality \refeq{eq-YangSym} gives
\begin{equation}
M_{022230}\Yang{qt^2}3=\frac{1-qt^2}{1-qt^3}M_{022230}.
\end{equation}
Similarly $\Den{032220}=\frac{1-qt}{1-qt^2}\Den{023220}$. And finally, $\Den{032220}=(1-qt)\Den{022230}$.\\
The reader should be aware that jumping over a block does not necessarily multiply the denominator by a factor of the form $1-q^{a} t^{b}$, but in general by a fraction whose numerator is $1-q^{a} t^{b}$. This becomes quite clear when one looks at the steps of the process described above. For instance, we have
\begin{equation}
\Den{032220}=\frac{1-qt}{1-qt^3}\Den{022230}.
\end{equation}
The same phenomenon also occurs in cases where the explanation is less straightforward. For instance:
\begin{equation}
\Den{021110}=\frac{1-qt}{1-qt^4}\Den{011120}.
\end{equation}
The formulation of the next result takes this  subtlety into consideration.
\begin{lem}[Jumping Lemma]\label{l-jl}
For any $b>a\geq 0$, $u'\in\mathbb N^m$, $u''\in \mathbb N^n$, and $k>0$, the polynomial $\num{\frac{\Den{u'ba^ku''}}{\Den{u'a^kbu''}}}$  divides $1-q^{b-a}t^{\beta-\alpha-k+1}$ where $\spectre{u'a^kbu''}[m+k]=q^at^{\alpha}$ and
$\spectre{u'a^kbu''}[m+k+1]=q^bt^{\beta}$.
\end{lem}
\begin{proof}
The result is proved by induction on $k$. In the case $k = 1$, the assertion is immediate, following directly from the rule derived from the Yang--Baxter graph through the action of the operator $\Yang{q^{a-b}t^{\alpha-\beta}}{m+1}$.\\
Now we assume that $k>1$. Using the induction hypothesis, we obtain that
$\num{\frac{\Den{u'aba^{k-1}u''}}{\Den{u'a^kbu''}}}$ divides $1-q^{b-a}t^{\beta-\alpha-k+2}$.
Furthermore
\begin{equation}\label{eq-slr-step-1}\begin{array}{rcl}
M_{u'aba^{k-1}u''}&=&M_{u'a^2ba^{k-2}u''}\Yang{q^{b-a}t^{\beta-\alpha-k+2}}{m+2}\\&=&M_{u'a^2ba^{k-2}u''}T_i+\frac{1-t}{1-q^{b-a}t^{\beta-\alpha-k+2}}M_{u'a^2ba^{k-2}u''},
\end{array}
\end{equation}
because
\begin{equation}
\left(\spectre{u'a^2ba^{k-2}u''}[m+1],\spectre{u'a^2ba^{k-2}u''}[m+2]\right)=
\left(q^at^{\alpha+k-2},q^bt^\beta\right).
\end{equation}
Again by induction, the polynomial $\num{\frac{\Den{u'a^2ba^{k-2}u''}}{\Den{u'a^{k}bu''}}}$ does not divide $1-q^{b-a}t^{\beta-\alpha-k+2}$ (because it divides $1-q^{b-a}t^{\beta-\alpha-k+3}$).  It follows that the denominator factor $1 - q^{b-a} t^{\beta-\alpha-k+2}$ of 
$M_{u'aba^{k-1}u''}$ appears exclusively in the term
\begin{equation}
\frac{1-t}{1 - q^{b-a} t^{\beta-\alpha-k+2}}\, M_{u'a^2ba^{k-2}u''}
\end{equation}
on the right-hand side of Equality~\ref{eq-slr-step-1}.
Since, by Proposition \ref{p-symM}, the Macdonald polynomial  $M_{u'a^2ba^{k-2}u''}$ is symmetric in the variables $x_{k+1}$ and $x_{k+2}$ one has
\begin{equation}\label{eq-u'a^2ba^(k-2)u''}
M_{u'a^2ba^{k-2}u''}\Yang{q^{b-a}t^{\beta -\alpha-k+1}}{m+1}=\frac{1-q^{b-a}t^{\beta -\alpha-k+2}}{1-q^{b-a}t^{\beta -\alpha-k+1}}.
\end{equation}
Combining Equality \refeq{eq-slr-step-1} and Equality \refeq{eq-u'a^2ba^(k-2)u''}, we obtain
\begin{equation}\begin{array}{rcl}
M_{u'ba^ku''}&=&M_{u'abu^{k-1}u''}\Yang{q^{b-a}t^{\beta-\alpha+k+1}}{m+1}\\
&=&
M_{u'a^2ba^{k-2}u''}T_i\Yang{q^{b-a}t^{\beta-\alpha-k+1}}{m+1}\\&&+\frac{(1-t)^2}{1-q^{b-a}t^{\beta-\alpha-k+1}}M_{u'a^2ba^{k-2}u''}.\end{array}
\end{equation}
We deduce that  $\num{\frac{\Den{u'baku''}}{\Den{u'a^kbu''}}}$ divides $(1-q^{b-a}t^{\beta-\alpha-k+1})(1-q^{b-a}t^{\beta-\alpha-k+3})$. But from the Yang-Baxter graphs it also divides  $(1-q^{b-a}t^{\beta-\alpha-k+1})(1-q^{b-a}t^{\beta-\alpha-k+2})$ because $M_{u'ba^ku''}=M_{u'aba^{k-1}u''}\Yang{q^{b-a}t^{\beta-\alpha-k+1}}{m+1}$. Hence, by the Disjonction rule, it divides \begin{equation}(1-q^{b-a}t^{\beta-\alpha-k+1})\gcd\left(1-q^{b-a}t^{\beta-\alpha-k+2},1-q^{b-a}t^{\beta-\alpha-k+3}\right)=1-q^{b-a}t^{\beta-\alpha-k+1}.\end{equation}
Indeed, since $\left(1-q^{b-a}t^{\beta-\alpha-k+2}\right)t-\left(1-q^{b-a}t^{\beta-\alpha-k+3}\right)=t-1$, any polynomials which divides both $1-q^{b-a}t^{\beta-\alpha-k+2}$ and $1-q^{b-a}t^{\beta-\alpha-k+3}$ also divides $1-t$. Nevertheless, $1 - t$ divides $1 - q^{c} t^{d}$ only when $1 - q^{c} = 0$; as $q$ is  formal, this condition forces $c = 0$. In our case, we have $c=b-a>0$. Hence, $\num{\frac{\Den{u'baku''}}{\Den{u'a^kbu''}}}$ divides $1-q^{b-a}t^{\beta-\alpha-k+1}$ as expected.
\end{proof}
Notice that this is just a consequence of the action of the Yang-Baxter operator on symmetric polynomials. Hence, one can apply almost the same reasoning to another configuration and find the following result whose proof is very similar to those of Lemma \ref{l-jl} and left to the reader.
\begin{lem}[Dual jumping Lemma]\label{l-djl}
For any $b>a\geq 0$, $u'\in\mathbb N^m$, $u''\in \mathbb N^n$, and $k>0$ the polynomial $\num{\frac{\Den{u'b^kau''}}{\Den{u'ab^ku''}}}$  divides $1-q^{b-a}t^{\beta-\alpha-k+1}$ where $\spectre{u'ab^ku''}[m+1]=q^at^{\alpha}$ and
$\spectre{u'ab^ku''}[m+2]=q^bt^{\beta}$.
\end{lem}
\begin{ex}
We have
\begin{equation}\Den{033320}=\frac{1-qt}{1-qt^2}\Den{033230}=\frac{1-qt}{1-qt^3}\Den{032330}=(1-qt)\Den{023330}.\end{equation}
\end{ex}
\section{The block-jump rule\label{sec-bloc}}
\subsection{Small jumps to big jumps}
It should be recalled that computing Macdonald polynomials along different paths on the Yang–Baxter graph yields the same result; the process is confluent. In particular, observe, in Figure \ref{f-022330to033220} what happens in the computation of $M_{033220}$ from $M_{022330}$. 
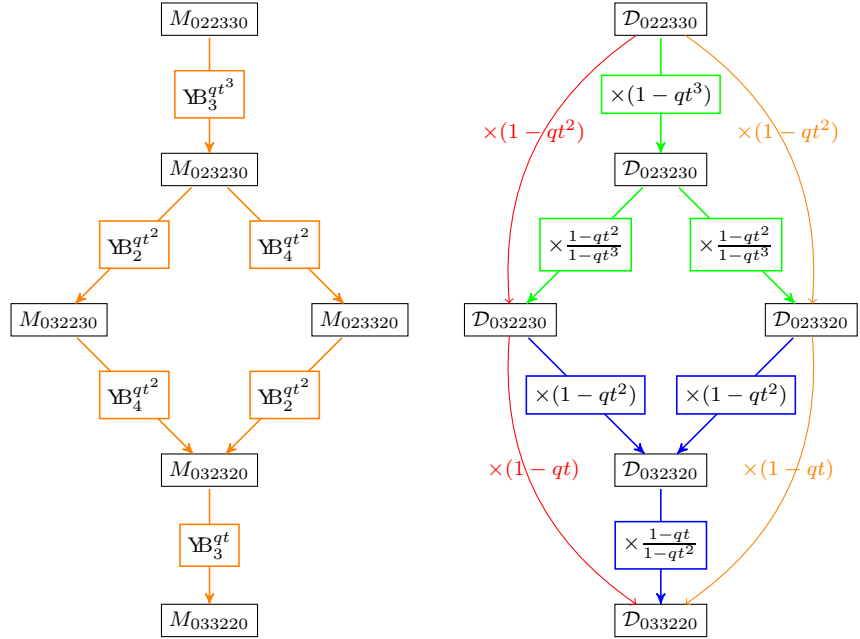
\begin{figure}[ht]
\begin{center}
\tiny
\begin{tikzpicture}
\GraphInit[vstyle=Shade]
    \tikzstyle{VertexStyle}=[shape = rectangle,
draw
]
\SetUpEdge[lw = 1.5pt,
color = orange,
 labelcolor = gray!30,
 style={post}
]
\tikzset{LabelStyle/.style = {draw,
                                     fill = white,
                                     text = black}}
\tikzset{EdgeStyle/.style={post}}
\Vertex[x=0, y=8,
 L={$M_{022330}$}]{x}
\Vertex[x=0, y=6,
 L={$M_{023230}$}]{y}
 \Vertex[x=-2, y=4,
 L={$M_{032230}$}]{z1}
 \Vertex[x=2, y=4,
 L={$M_{023320}$}]{z2}
 \Vertex[x=0, y=2,
 L={$M_{032320}$}]{t}
 \Vertex[x=0, y=0,
 L={$M_{033220}$}]{s}
\Edge[label={$\Yang{qt^3}3$}](x)(y)
\Edge[label={$\Yang{qt^2}2$}](y)(z1)
\Edge[label={$\Yang{qt^2}4$}](y)(z2)
\Edge[label={$\Yang{qt^2}4$}](z1)(t)
\Edge[label={$\Yang{qt^2}2$}](z2)(t)
\Edge[label={$\Yang{qt}3$}](t)(s)

\Vertex[x=6, y=8,
 L={$\Den{022330}$}]{fx}
\Vertex[x=6, y=6,
 L={$\Den{023230}$}]{fy}
 \Vertex[x=4, y=4,
 L={$\Den{032230}$}]{fz1}
 \Vertex[x=8, y=4,
 L={$\Den{023320}$}]{fz2}
 \Vertex[x=6, y=2,
 L={$\Den{032320}$}]{ft}
 \Vertex[x=6, y=0,
 L={$\Den{033220}$}]{fs}
\Edge[color={green},label={$\times (1-qt^3)$}](fx)(fy)
\Edge[color={green},label={$\times\frac{1-qt^2}{1-qt^3}$}](fy)(fz1)
\Edge[color = {green}, label={$\times\frac{1-qt^2}{1-qt^3}$}](fy)(fz2)
\Edge[color={blue},label={$\times (1-qt^2)$}](fz1)(ft)
\Edge[color = {blue}, label={$\times(1-qt^2)$}](fz2)(ft)
\Edge[color={blue},label={$\times\frac{1-qt}{1-qt^2}$}](ft)(fs)
\draw[->, bend right, red] (fx) to node[above] {$\times(1-qt^2)$} (fz1);
\draw[->, bend right, red] (fz1) to node[above] {$\times(1-qt)$} (fs);

\draw[->, bend left, orange] (fx) to node[above] {$\times(1-qt^2)$} (fz2);
\draw[->, bend left, orange] (fz2) to node[above] {$\times(1-qt)$} (fs);
\end{tikzpicture}
\end{center}
\caption{From $M_{022330}$ to $M_{33220}$. \label{f-022330to033220}}
\end{figure}
And, finally, $\Den{033220}=(1-qt)(1-qt^2)\Den{022330}$.
This can be obtained by applying Jump Lemma (Lemma \ref{l-jl}) two times, following the red path (see Figure \ref{f-022330to033220}).
This method can be easily generalized as follows.
\begin{prop}\label{p-block}
For any $b>a\geq 0$, $u'\in\mathbb N^m$, $u''\in \mathbb N^n$, and $k,\ell>0$ the polynomial $\num{\frac{\Den{u'b^\ell a^ku''}}{\Den{u'a^k b^\ell u''}}} $ divides
$\displaystyle\prod_{i=k-1}^{k+\ell-2}\big(1-q^{b-a}t^{\beta-\alpha-i}\big),$
where
$\spectre{u'a^kb^\ell u''}[m+k]=q^at^{\alpha}$ and
$\spectre{u'a^kb^\ell u''}[m+k+1]=q^bt^{\beta}$.
\end{prop}
\begin{proof}
We proceed by induction on $\ell$. If $\ell = 1$ the result is given by Lemma \ref{l-jl}. If $\ell>1$, we first apply Lemma \ref{l-jl} and show that $\num{\frac{\Den{u'ba^{k}b^{\ell -1}u''}}{\Den{u'a^kb^\ell u''}}}$ divides $1-q^{b-a}t^{\beta-\alpha-k+1}$. We apply the induction hypothesis to $u'ba^kb^{\ell'}u''=va^kb^{\ell'}u''$ with $\spectre{va^kb^\ell u''}[m'+k] = q^at^{\alpha}$ and 
$\spectre{va^kb^\ell u''}[m'+k+1] = q^at^{\beta'}$ where $\ell'=\ell-1$, $v=u'b$, $m'=m+1$, and $\beta'=\beta-1$. We find that $\num{\frac{\Den{u'b^\ell a^ku''}}{\Den{u'ba^k b^{\ell-1}u''}}}$ divides
\begin{equation}\begin{array}{c}
(1-q^{b-a}t^{\beta'-\alpha-k+1})(1-q^{b-a}t^{\beta'-\alpha-k})\cdots(1-q^{b-a}t^{\beta'-\alpha-k-\ell'+2})=\\
(1-q^{b-a}t^{\beta-\alpha-k})(1-q^{b-a}t^{\beta-\alpha-k-1})\cdots(1-q^{b-a}t^{\beta-\alpha-k-\ell+2}).
\end{array}
\end{equation}
Since  \begin{equation}\frac{\Den{u'b^\ell a^{k}b^{\ell -1}u''}}{\Den{u'a^kb^\ell u''}}= \frac{\Den{u'b^\ell a^{k}u''}}{\Den{u'ba^kb^{\ell-1} u''}}\frac{\Den{u'ba^kb^{\ell-1} u''}}{\Den{u'a^kb^{\ell} u''}},\end{equation} the Conjonction rule implies that 
$\num{\frac{\Den{u'b^\ell a^{k}b^{\ell -1}u''}}{\Den{u'a^kb^\ell u''}}}$ divides 
\begin{equation}(1-q^{b-a}t^{\beta-\alpha-k+1})(1-q^{b-a}t^{\beta-\alpha-k})\cdots(1-q^{b-a}t^{\beta-\alpha-k-\ell+2}),\end{equation}
as expected.
\end{proof}
It should be noted that $M_{033220}$ may also be derived by following the orange arrows shown in Figure~\ref{f-022330to033220}, which corresponds to applying the dual jump lemma twice.\\
This suggests a dual version of Proposition~\ref{p-block} whose proof is very similar.
\begin{prop}\label{p-blockdual}
For any $b>a\geq 0$, $u'\in\mathbb N^m$, $u''\in \mathbb N^n$, and $k,\ell>0$ the polynomial $\num{\frac{\Den{u'b^\ell a^ku''}}{\Den{u'a^k b^\ell u''}}} $  divides
$\displaystyle\prod_{i=\ell-1}^{k+l-2}\big(1-q^{b-a}t^{\beta-\alpha-i}\big)$,
where
$\spectre{u'a^kb^\ell u''}[m+k]=q^at^{\alpha}$ and
$\spectre{u'a^kb^\ell u''}[m+k+1]=q^bt^{\beta}$.
\end{prop}
The following Theorem summarizes  Propositions \ref{p-block} and \ref{p-blockdual}.
\begin{thm}[Block jump rule]\label{thm-block}
For any $b>a\geq 0$, $u'\in\mathbb N^m$, $u''\in \mathbb N^n$, and $k,\ell>0$, the polynomial $\num{\frac{\Den{u'b^\ell a^ku''}}{\Den{u'a^k b^\ell u''}}} $  divides
$\displaystyle\prod_{i=\max\{k,\ell\}-1}^{k+\ell-2}\big(1-q^{b-a}t^{\beta-\alpha-i}\big),$
where
$\spectre{u'a^kb^\ell u''}[m+k]=q^at^{\alpha}$ and
$\spectre{u'a^kb^\ell u''}[m+k+1]=q^bt^{\beta}$.
\end{thm}
\begin{proof}
It suffices to apply the Disjonction rule and deduce that $\num{\frac{\Den{u'b^\ell a^ku''}}{\Den{u'a^k b^\ell u''}}}$ divides the greates common divisor of the polynomials
\begin{equation}(1-q^{b-a}t^{\beta-\alpha-k+1})(1-q^{b-a}t^{\beta-\alpha-k})\cdots(1-q^{b-a}t^{\beta-\alpha-k-\ell+2})\end{equation}
and
\begin{equation}(1-q^{b-a}t^{\beta-\alpha-\ell+1})(1-q^{b-a}t^{\beta-\alpha-\ell})\cdots(1-q^{b-a}t^{\beta-\alpha-k-\ell+2}).\end{equation}
\end{proof}

As an example, we obtain $M_{0333220}$ from $M_{0223330}$ by following two paths (see Figure \ref{f-0223330to0333220}) either we apply three jumps (red path)
\begin{equation}
{0223330}\mathop{\longrightarrow} {0322330}\longrightarrow {0332230}\longrightarrow M_{0333220},
\end{equation}
either we apply  two dual jumps (orange path)
\begin{equation}
{0223330}\mathop{\longrightarrow} {0233320}\longrightarrow  {0333220}.
\end{equation}
The red path gives that $\num{\frac{\Den{0333220}}{\Den{0223330}}}$ divides $(1-qt^3)(1-qt^2)(1-qt)$.\\
But the orange path gives that  $\num{\frac{\Den{0333220}}{\Den{0223330}}}$ divides $(1-qt^2)(1-qt)$. The Disjonction rule  implies that  $\num{\frac{\Den{0333220}}{\Den{0223330}}}$ divides \begin{equation}\gcd((1-qt^3)(1-qt^2)(1-qt),(1-qt^2)(1-qt))=(1-qt^2)(1-qt).\end{equation}
\begin{figure}[ht]
\begin{center}
\tiny
\begin{tikzpicture}
\GraphInit[vstyle=Shade]
    \tikzstyle{VertexStyle}=[shape = rectangle,
draw
]
\SetUpEdge[lw = 1.5pt,
color = orange,
 labelcolor = gray!30,
 style={post}
]
\tikzset{LabelStyle/.style = {draw,
                                     fill = white,
                                     text = black}}
\tikzset{EdgeStyle/.style={post}}
\Vertex[x=0, y=8,
 L={$M_{0223330}$}]{x}
\Vertex[x=-2, y=6,
 L={$M_{0322330}$}]{y}
 \Vertex[x=-2, y=4,
 L={$M_{0332230}$}]{z1}
 \Vertex[x=2, y=5,
 L={$M_{0233320}$}]{z2}
 \Vertex[x=0, y=2,
 L={$M_{0333220}$}]{t}
\Edge[color={red},label={$\Yang{qt^4}3\Yang{qt^3}2$}](x)(y)
\Edge[color={red},label={$\Yang{qt^3}4\Yang{qt^2}3$}](y)(z1)
\Edge[color = {red},label={$\Yang{qt^2}5\Yang{qt}4$}](z1)(t)
\Edge[color = {orange},label={$\Yang{qt^4}3\Yang{qt^3}4\Yang{qt^2}5$}](x)(z2)
\Edge[color = {orange},label={$\Yang{qt^3}2\Yang{qt^2}3\Yang{qt}4$}](z2)(t)

\Vertex[x=6, y=8,
 L={$\Den{0223330}$}]{fx}
\Vertex[x=4, y=6,
 L={$\Den{0322330}$}]{fy}
 \Vertex[x=4, y=4,
 L={$\Den{0332230}$}]{fz1}
 \Vertex[x=8, y=5,
 L={$\Den{0233320}$}]{fz2}
 \Vertex[x=6, y=2,
 L={$\Den{0333220}$}]{t}
\draw[->, bend right, red] (fx) to node[above] {$\times(1-qt^3)$} (fy);
\draw[->,  red] (fy) to node[above] {$\times(1-qt^2)$} (fz1);
\draw[->, bend right, red] (fz1) to node[above] {$\times(1-qt)$} (t);

\draw[->, bend left, orange] (fx) to node[above] {$\times(1-qt^2)$} (fz2);
\draw[->, bend left, orange] (fz2) to node[above] {$\times(1-qt)$} (ft);
\end{tikzpicture}
\end{center}
\caption{From $M_{022330}$ to $M_{33220}$. \label{f-0223330to0333220}}
\end{figure}
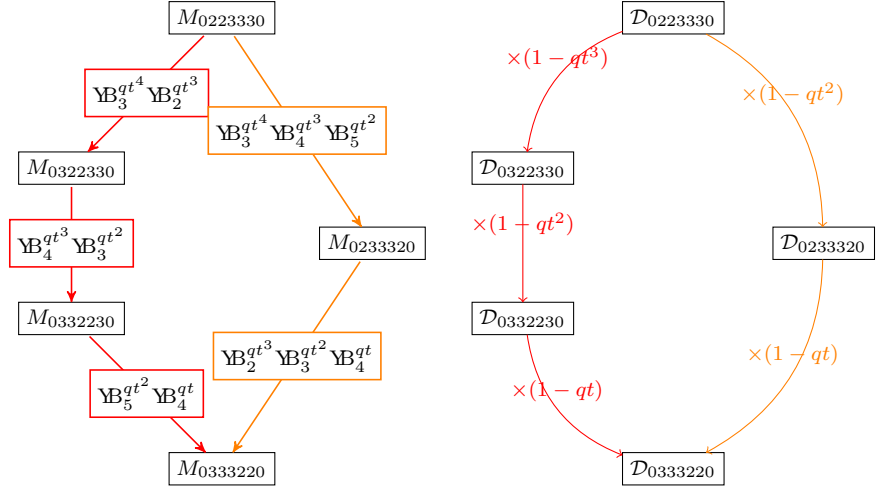
We called $\jump$ the algorithm such that,
for every path $u'a^kb^\ell u''\xrightarrow[]{\mathfrak p}u'b^\ell a^ku''$,
\begin{equation}\jump(\mathfrak p)=(1-q^{b-a}t^{\beta-\min\{k,\ell\}-\ell+1})(1-q^{b-a}t^{\beta-\alpha-\ell})\cdots(1-q^{b-a}t^{\beta-\alpha-k-\ell+2}),\end{equation}  where $u'$ is a vector of length $m$,
$\spectre{u'a^kb^\ell u''}[m+k]=q^at^{\alpha}$ and
$\spectre{u'a^kb^\ell u''}[m+k+1]=q^bt^{\beta}$, and $\jump(\mathfrak p)=\triv(\mathfrak p)$ elsewhere.
We will consider also the special paths $u'a^kb^\ell u''\xrightarrow[(m+1;k,\ell)]{\jumppath}u'b^\ell a^ku''$, which means that acting by 
\begin{equation}\big(\Yang{q^{b-a}t^{\beta-\alpha}}{m+k}\cdots \Yang{q^{b-a}t^{\beta-\alpha-k+1}}{m+1}\big)\cdots
\big(\Yang{q^{b-a}t^{\beta-\alpha-\ell+1}}{m+k+\ell-1}\cdots \Yang{q^{b-a}t^{\beta-\alpha-k-\ell+2}}{m+\ell}\big)
\end{equation}
on $M_{u'a^kb^\ell u''}$ produces $M_{u'b^\ell a^k u''}$
and
$u'a^kb^\ell u''\xrightarrow[(m+1,k,\ell)]{\djumppath}u'b^\ell a^ku''$, which means that acting  by 
\begin{equation}\big(\Yang{q^{b-a}t^{\beta-\alpha}}{m+k}\cdots \Yang{q^{b-a}t^{\beta-\alpha-\ell+1}}{m+k+\ell-1}\big)\cdots
\big(\Yang{q^{b-a}t^{\beta-\alpha-k+1}}{m+1}\cdots \Yang{q^{b-a}t^{\beta-\alpha-k-\ell+2}}{m+\ell}\big)
\end{equation}
on $M_{u'a^kb^\ell u''}$ produces $M_{u'b^\ell a^ku''}$.\\
Note that, for the sake of readability, we only display certain parameters:  
the position of the first element of the first block ($m+1$),
the size of the first block ($k$),  
and the size of the second block ($\ell$). The remaining parameters $a$, $b$, $\alpha$, and $\beta$ can be recovered from the examination of the vector $u' a^k b^\ell u''$.
\\
For instance, in Figure \ref{f-0223330to0333220}
the path ${0223330}{\color{red}\xrightarrow[(2,2,3)]{\jumppath}}{0333220}$ is drawn in red correspond to the operator
\begin{equation}
\big(\Yang{qt^{4}}3\Yang{qt^3}2\big)\big(\Yang{qt^3}4\Yang{qt^2}3\big)
\big(\Yang{qt^2}5\Yang{qt}4\big),
\end{equation}
because $\spectre{0223330}=[t,t^3q^2,t^2q^2,t^6q^3,t^5q^3,t^4q^3,1].$\\
The path
${0223330}{\color{orange}\xrightarrow[(2,2,3)]{\djumppath}}{0333220}$ corresponding to the operator
\begin{equation}
\big(\Yang{qt^4}3\Yang{qt^3}4\Yang{qt^2}5\big)\big(\Yang{qt^3}2\Yang{qt^2}3\Yang{qt}4\big)
\end{equation} is drawn in orange.\\
The paths $\jumppath(i;k,\ell)$ and $\djumppath(i;k,\ell)$ are distinct, but thanks to the confluence of the Yang–Baxter graph, they represent the same operator $\jumpop(i;k,\ell)$ acting on Macdonald polynomials with two consecutive blocks, the first starting at position $i$ with size $k$, and the second of size $\ell$.
The definitions of $\jumppath(i;k,\ell)$ and $\djumppath(i;k,\ell)$ thus provide two different expressions for the same operator $\jumpop(i;k,\ell)$.\\
More precisely, the path $\jumppath(i;k,\ell)$ (resp. $\djumppath(i;k,\ell)$) is defined as the concatenation of $\ell$ (resp. $k$) elementary jumps, namely 
\begin{equation}
\jumppath(i;k,1)\cdots \jumppath(i+\ell-1;k,1);
\end{equation}
respectively
\begin{equation}
\djumppath(i+k-1;1,\ell)\cdots \djumppath(i;1,\ell).
\end{equation}
Hence, we have
\begin{equation}
\begin{array}{rcl}
\jumpop(i;k,\ell)
&=&
\jumpop(i;k,1)\cdots \jumpop(i+\ell-1;k,1)
\\[4pt]
&=&
\jumpop(i+k-1;1,\ell)\cdots \jumpop(i;1,\ell).
\end{array}
\end{equation}
To avoid any confusion, we emphasize that the notion of a jump in the Yang–Baxter graph gives rise to four distinct but related objects : 
\begin{itemize}
    \item two paths, $\jumppath$ and 
    $\djumppath$;
    \item a partially defined operator $\jumpop$
    on Macdonald polynomials, corresponding to the operation performed when following these paths;
    \item and an algorithm $\jump$ that computes a denominator using the properties of the paths 
    $\jumppath$ and $\djumppath$.
\end{itemize}
\subsection{Explicit expressions for jumps}
The next proposition gives a concise expression for the elementary jumps.
\begin{prop}\label{prop-elemjumps}
Let $b>a$ and $u'$ be a vector of size $m$.
\begin{enumerate}
    \item 
Let $v=u'a^kbu''$. Suppose that the spectral vector satisfies $\frac{\spectre{v}[m+k+1]}{\spectre{v}[m+k]}=t^\alpha q^{b-a}$ then $\jumpop(m+1;k,1)$ acts on $M_v$ as the operator
\begin{equation}\label{eq-jumopformula}
J^{q^{b-a}t^{\alpha-k+1}}_{m+1,k}=T_{m+k}\cdots T_{m+1}+\frac{1-t}{1-q^{b-a}t^{\alpha-k+1}}\left(1+\sum_{i=2}^{k}T_{m+k}\cdots T_{m+i}\right).
\end{equation}
\item Let $w=u'ab^\ell u''$.
Suppose that the spectral vector satisfies $\frac{\spectre{w}[m+1]}{\spectre{w}[m+2]}=t^\alpha q^{b-a}$ then $\jumpop(m+1;1,\ell)$ acts on $M_w$ as
\begin{equation}\label{eq-djumopformula}
{J^\dag}^{q^{b-a}t^{\alpha-\ell+1}}_{m+1,\ell}=T_{m+1}\cdots T_{m+\ell}+\frac{1-t}{1-q^{b-a}t^{\alpha-\ell+1}}\left(1+\sum_{i=1}^{\ell-1}T_{m+1}\cdots T_{m+i}\right).
\end{equation}
\end{enumerate}
\end{prop}
\begin{proof}
The two assertions can be proved in a similar manner; we  only present the proof of the first one. We proceed by induction on $k$. For $k = 1$, the result follows directly from the definition of $\Yang{q^{b-a}t^\alpha}{m+1}$. Assume that the statement holds for $k-1$.  
The polynomial $M_{u'aba^{k-1}u''}$ is obtained by applying the operator
\begin{equation}
T_{m+k}\cdots T_{m+2}
+\frac{1-t}{1-q^{b-a}t^{\alpha-k+2}}
\left(1+\sum_{i=2}^{k-1}T_{m+k}\cdots T_{m+i}\right)
\end{equation}
to the polynomial $M_v$.
Since $M_{u'ba^ku''}=M_{u'aba^{k-1}u''}\cdot\Yang{q^{b-a}t^{\alpha-k+1}}{m+1}$, we have to understand the action of the operator
\begin{equation}\left(
T_{m+k}\cdots T_{m+2}
+\frac{1-t}{1-q^{b-a}t^{\alpha-k+2}}
\left(1+\sum_{i=2}^{k-1}T_{m+k}\cdots T_{m+i+1}\right)\right)\cdot\Yang{q^{b-a}t^{\alpha-k+1}}{m+1}
\end{equation}
on $v$. The polynomial $M_v$ is symmetric in the variables $x_{m+1}$ and $x_{m+2}$.  

Since the operator 
$\displaystyle
\left(1+\sum_{i=2}^{k-1}T_{m+k}\cdots T_{m+i+1}\right)
$
does not act on these variables, the polynomial 
$\displaystyle
M_v\cdot\left(1+\sum_{i=2}^{k-1}T_{m+k}\cdots T_{m+i+1}\right)
$
is also symmetric in the variables $x_{m+1}$ and $x_{m+2}$. Hence by \refeq{eq-YangSym}, we have
\begin{equation}\begin{array}{l}
M_v\cdot\left(\displaystyle1+\sum_{i=2}^{k-1}T_{m+k}\cdots T_{m+i+1}\right)\Yang{q^{b-a}t^{\alpha-k+1}}{m+1}=\\\quad\quad\quad\quad\quad\quad\quad\quad\quad\displaystyle
\frac{1-q^{b-a}t^{\alpha-k+2}}{1-q^{b-a}t^{\alpha-k+1}}
 M_v\cdot\left(1+\sum_{i=2}^{k-1}T_{m+k}\cdots T_{m+i+1}\right).\end{array}
\end{equation}
Hence,
\begin{equation}\begin{array}{rcl}
M_v\cdot\jump(m+1;k,1)&=&\displaystyle M_v\cdot T_{m+k}\cdot T_{m+2}\left(T_{m+1}+\frac{1-t}{1-q^{b-a}t^{\alpha-k+1}}\right)+
\\&&\displaystyle\frac{1-t}{1-q^{b-a}t^{\alpha-k+1}}M_v\cdot\left(1+\sum_{i=2}^{k-1}T_{m+k}\cdots T_{m+i+1}\right)\\
&=&M_v\cdot T_{m+k}\cdots T_{m+1}+\\&&\displaystyle\frac{1-t}{1-q^{b-a}t^{\alpha-k+1}}M_v\cdot\left(1+\sum_{i=2}^{k}T_{m+k}\cdots T_{m+i}\right).
\end{array}
\end{equation}
\end{proof}

\begin{rem}\rm
Assume that $a$, $b$, and $v$ satisfy the conditions of Proposition~\ref{prop-elemjumps}.1.
The reader should be careful that Proposition~\ref{prop-elemjumps} does not imply that 
$
\frac{\Den{u'ba^ku''}}{\Den{v}} = 1 - q^{b-a}t^{\alpha-k+1},
$
but only that 
$
\num{\frac{\Den{u'ba^ku''}}{\Den{v}}} = 1 - q^{b-a}t^{\alpha-k+1}$.
Indeed, some poles may cancel between the two vectors.\\
For instance, we have $\Den{0022}=q^2(1-qt)(1-qt^2)$ while $\Den{2002}=q^2(1-qt)(1-q^2t^2)$. The factor $1 - qt^2$ vanishes during the jump.
\end{rem}
\begin{cor}
Let $w=u'a^kb^\ell u''$ with $u'\in\mathbb N^m$,  $b>a$, and $\frac{\spectre w[m+k+1]}{\spectre w[m+k]}=t^\alpha q^{b-1}$. The following operator identities are satisfied when acting on $M_w$ :
\begin{equation}
\begin{array}{rcl}
\jumpop(m+1,k,\ell)&=&J_{m+1,k}^{q^{b-a}t^{\alpha-k+1}}
J_{m+2,k}^{q^{b-a}t^{\alpha-k}}\cdots J_{m+\ell,k}^{q^{b-a}t^{\alpha-k-\ell+2}}\\
&=&
{J^\dag}_{m+k,\ell}^{q^{b-a}t^{\alpha-\ell+1}}
{J^\dag}_{m+k-1,\ell}^{q^{b-a}t^{\alpha-\ell}}\cdots {J^\dag}_{m+1,\ell}^{q^{b-a}t^{\alpha-k-\ell+2}}.
\end{array}
\end{equation}
\end{cor}
\begin{proof}
Using the confluence property of the Yang--Baxter graph, the result follows from a simple induction and Proposition \ref{prop-elemjumps}.

\end{proof}
\section{Jumping in the staircases\label{sec-staircase}}
A staircase Macdonald polynomial is a Macdonald polynomial indexed by
\begin{equation}
\staircase(k, a, n) = ((n-1)a)^k\cdot((n-2)a)^k\cdots a^k\cdot 0^k,
\end{equation}
with the notation $a^k = \underbrace{a \cdots a}_{k\times}$.

\subsection{Staircase Macdonald polynomials from scratch}
In order to compute the staircase Macdonald polynomials, we shall consider Macdonald polynomials indexed by intermediate shapes defined by
\begin{equation}
\Qsc(k,a,n;m,b)
 = ((m-1)a+b)^k\cdot((m-2)a+b)^k\cdots(a+b)^k\cdot b^k\cdot 0^{k(n-m)},
\end{equation}
for all $a, k, n > 0$, $1 \leq m \leq n$, and $1 \leq b \leq a$.\\
Their definition directly yields an inductive computation procedure. Indeed,
If $1 \leq m \leq n$ and $1 \leq b < a$, we have
\begin{equation}\label{eq-Vi}\begin{array}{ll}
\Qsc(k,a;m,b)\xrightarrow[]{\affinepath^{mk}} V_0& \xrightarrow[(1;(n-m)k,k)]{\jumppath} V_1
\xrightarrow[(k+1;(n-m)k,k)]{\jumppath} V_2\cdots\\&\quad\quad\quad\cdots
\xrightarrow[((m-1)k+1;(n-m)k,k)]{\jumppath}{\Qsc(k,a;m,b+1)}\end{array}
\end{equation}
with \begin{equation}V_0=0^{k(n-m)}\cdot((m-1)a+b+1)^k\cdot ((m-2)a+b+1)^k\cdots(a+b+1)^k\cdot (b+1)^k,\end{equation}
and for any $0<i\leq m$,
\begin{equation}
V_i=((m-1)a+b+1)^k\cdots((m-i)a+b+1)^k\cdot0^{k(n-m)}\cdot((m-i-1)a+b+1)^k\cdots (b+1)^k.
\end{equation}
This path will be denoted by
\begin{equation}
\Qsc(k,a;m,b)\xrightarrow[(k,a;m,b)]{\rai}\Qsc(k,a;m,b+1)
\end{equation}
If $1 \leq m < n$, we have
\begin{equation}\label{eq-wi}\begin{array}{ll}
{\Qsc(k,a;m,a)}\xrightarrow[]{\affinepath^{(m+1)k}} W_0& \xrightarrow[(1;(n-m+1)k,k)]{\jumppath} W_1
\xrightarrow[(k+1;(n-m+1)k,k)]{\jumppath} W_2\cdots\\&\quad\quad\quad\cdots
\xrightarrow[(mk+1;(n-m+1)k,k)]{\jumppath}{\Qsc(k,a;m+1,1)}\end{array}
\end{equation}
with
\begin{equation}
W_0 = 0^{k(n-m-1)}\cdot(ma+1)^k\cdot((m-1)a+1)^k\dots1^k,
\end{equation}
and for any $0<i<m$,
\begin{equation} W_i = (ma+1)^k\dots((m-i+1)a+1)^k\cdot0^{k(n-m-1)}\cdot
((m-i)a+1)^k\cdots1^k. \end{equation}
This path will be denoted by
\begin{equation}
\Qsc(k,a;m,a)\xrightarrow[(k,a;m)]{\addstep}
\Qsc(k,a;m+1,1).
\end{equation}
For instance, we obtain ${\Qsc(2,2,3;2,2)}=44220000$ from ${\Qsc(2,2,3;2,1)}=33110000$ by
\begin{equation}
{33110000}\xrightarrow[]{\affinepath^4}{00004422}\xrightarrow[(1;4,2)]{\jumppath}{44000022}
\xrightarrow[(3;4,2)]{\jumppath}{44220000}
\end{equation}
and ${\Qsc(2,2,3;3,1)}={[55331100]}$ from 
${\Qsc(2,2,3;2,2)}={[44220000]}$ by
\begin{equation}
{44220000}\xrightarrow[]{\affinepath^6}{00553311}\xrightarrow[(1;2,2)]{\jumppath}{55003311}
\xrightarrow[(3;2,2)]{\jumppath}{55330011}
\xrightarrow[(5;2,2)]{\jumppath}{55331100}.
\end{equation}
We also consider the path
\begin{equation}\begin{array}{l}
\Qsc(k,a;m,a)\xrightarrow[(k,a;m)]{\addstep}
\Qsc(k,a;m+1,0)\xrightarrow[(k,a;m+1,0)]{\rai}
\cdots\\
\quad\quad\quad\quad\quad\quad\quad\quad\quad\quad\quad\quad\cdots\xrightarrow[(k,a;m+1,a-1)]{\rai} \Qsc(k,a;m+1,a).
\end{array}
\end{equation}
This path will be denoted by
\begin{equation}
\Qsc(k,a;m,a)\xrightarrow[(k,a;m)]{\pathup}\Qsc(k,a;m+1,a).
\end{equation}
The staircase vector 
$\staircase(k;a,n)$ is obtained by successively following paths of the form $\xrightarrow[(k,a;m)]{\pathup}$:
\begin{equation}\begin{array}{rcl}
\Qsc(k,a,n;0,a)=0^{nk}&\xrightarrow[(k,a;0)]{\pathup}&\Qsc(k,a,n;1,a)=a^k0^{(n-1)k}\\
\Qsc(k,a,n;1,a)&\xrightarrow[(k,a;1)]{\pathup}&\Qsc(k,a,n;2,a)=(2a)^k\cdot a^k0^{(n-2)k}\\
&\vdots&\\
\Qsc(k,a,n;n-2,a)&\xrightarrow[(k,a;n-2)]{\pathup}&\Qsc(k,a,n;n-1,a)\\&&=\staircase(k,a,n).
\end{array}\end{equation}
For instance, we have
\begin{equation}
\begin{array}{rclcl}
000000000000&\xrightarrow[(3,2;0)]{\pathup}&222000000000&=&\Qsc(3,2,4;1,2)\\
222000000000&\xrightarrow[(3,2;1)]{\pathup}&444222000000&=&\Qsc(3,2,4;2,2)\\
444222000000&\xrightarrow[(3,2;2)]{\pathup}&666444222000
&=&\Qsc(3,2,4;3,2)\\
&&&=&\staircase(3,2,3)\end{array}
\end{equation}
In terms of operators, we define
\begin{equation}
\raiseop(k,a;m,b)=\affineop^{mk}\jumpop(1;(n-m)k,k)\cdots\jumpop((m-1)k+1;(n-m),k),
\end{equation}
\begin{equation}
\addop(k,a;m)=\affineop^{(m+1)k}\jumpop(1;(n-m+1)k,k)\cdots
\jumpop(mk+1;(n-m+1)k,k),
\end{equation}
and
\begin{equation}
\upop(k,a;m)=\addop(k,a;m)\raiseop(k,a;m+1,0)\cdots\raiseop(k,a;m+1,a-1).
\end{equation}
Hence we have
\begin{equation}\label{eq-genQsc1}
M_{\Qsc(k,a;m,b+1)}=M_{\Qsc(k,a;m,b)}\raiseop(k,a;m,b),
\end{equation}
\begin{equation}\label{eq-genQsc2}
M_{\Qsc(k,a;m+1,1)}=M_{\Qsc(k,a;m,a)}\addop(k,a;m),
\end{equation}
\begin{equation}\label{eq-genQcs3}
M_{\Qsc(k,a;m,a+1)}=M_{\Qsc(k,a;m,a)}\upop(k,a;m,a),
\end{equation}
and, finally,
\begin{equation}\label{eq-genstaircase}
M_{\staircase(k;a,n)}=1\cdot \upop(k,a;0)\upop(k,a;1)\cdots \upop(k,a;n-2).
\end{equation}

\subsection{The unreachable pole}
Let $1\leq j\leq m$, using the notation of \refeq{eq-Vi}, we have
\begin{equation}
\spectre{V_{j-1}}[(j-1)k+1]=t^{(n-m)k-1}\end{equation} and \begin{equation}\spectre{V_{j-1}}[(n-m+j)k+1]=q^{(m-j)a+b+1}t^{(n-j)k-1}
\end{equation}
Hence applying Theorem \ref{thm-block}, we obtain that $\num{\frac{D_{V_j}}{D_{V_{j-1}}}}$ divides \begin{equation}
D_j=\prod_{i=1}^k\big(1-q^{(m-j)a+b+1}t^{(m-j)k+i}\big).
\end{equation}
Following Proposition \ref{prop-specomega}, we have to prove that for any $1\leq i\leq k$, 
the polynomial $1-q^{(m-j)a+b+1}t^{(m-j)k+i}$ does not vanish for $(q,t)=(\omega u^{-\frac{1+k} d},u^{\frac ad})$, where $d=\gcd(a,1+k)$ and $\omega$ is a primitive $a$th-root of the unity. We have
\begin{equation}
1-\big({\omega}u^{-\frac{1+k}d}\big)^{(m-j-1)a+b+1}(u^{\frac ad})^{(m-j)k+i}=1-\omega^{b+1}u^{-\frac ad(m-j)-\frac{(b+1)(k+1)}d+\frac{ai}d} = 0,
\end{equation}
if and only if $a$ divides $b+1$ (this implies $a=b+1$ because $b<a$) and $-a(m-j)-(b+1)(k+1)+ai=0$ so  $i=m-j+(k+1)>k$.
Hence,  $1-q^at^{k+1}$ does not divides $D_j$ and, therefore, it does not divides  $\num{\frac{D_{V_j}}{D_{V_{j-1}}}}$.
From \refeq{eq-genQsc1}, and using successively many times the Conjonction rule (see Section \ref{subsec-denom}), we obtain the following lemma.
\begin{lem}\label{lem-divQsc1}
The polynomial  $1-q^at^{k+1}$ does not divides $\num{\frac{\Den{\Qsc(k,a;m,b+1)}}{\Den{\Qsc(k,a;m,b)}}}.$
\end{lem}
We proceed in the same way for Equality \refeq{eq-wi}. Let $1\leq j\leq m+1$. We have
\begin{equation}
\spectre{W_{j-1}}[(j-1)k+1]=t^{(n-m-1)k-1}
\end{equation}
and
\begin{equation}
\spectre{W_{j-1}}[(n-m+j-1)k+1]=q^{(m-j+1)a+1}t^{(n-j)k-1}.
\end{equation}
Applying Theorem \ref{thm-block}, we obtain that $\num{\frac{D_{w_j}}{D_{w_{j-1}}}}$ divides
\begin{equation}
E_j=\prod_{i=1}^k\big(1-q^{(m-j)a+1}t^{(m-j)k+i}\big).
\end{equation}
Following Proposition \ref{prop-specomega}, we have to prove that for any $1\leq i\leq k$, 
the polynomial $1-q^{(m-j)a+1}t^{(m-j)k+i}$ does not vanish for $(q,t)=(\omega u^{-\frac{1+k} d},u^{\frac ad})$, where $d=\gcd(a,1+k)$ and $\omega$ is a primitive $a$th-root of the unity. We have
\begin{equation}
1-\big({\omega}u^{-\frac{1+k}d}\big)^{(m-j)a+1}(u^{\frac ad})^{(m-j)k+i}=1-\omega u^{-\frac ad(m-j)-\frac{(k+1)}d+\frac{ai}d} \neq 0,
\end{equation}
Hence,  $1-q^at^{k+1}$ does not divides $E_j$ and, therefore, it does not divides  $\num{\frac{D_{W_j}}{D_{W_{j-1}}}}$. From \refeq{eq-genQsc2}, and using successively many times the Conjonction rule (see Section \ref{subsec-denom}), we obtain the following lemma.
\begin{lem}\label{lem-divQsc2}
The polynomial  $1-q^at^{k+1}$ does not divide $\num{\frac{\Den{\Qsc(k,a;m+1,1)}}{\Den{\Qsc(k,a;m,a)}}}.$
\end{lem}
From Lemmas \ref{lem-divQsc1} and \ref{lem-divQsc2} and formula \refeq{eq-genQcs3}, we deduce that $1-q^{a}t^{k-1}$ does not divide $\num{\frac{\Den{\Qsc(k,a;m,a+1)}}{\Den{\Qsc(k,a;m,a)}}},$
and then from \refeq{eq-genstaircase}, we deduce
\begin{thm}
The polynomial $1-q^{a}t^{k+1}$ does not divide
$\Den{\staircase(k,a,n)}$.
\end{thm}
\section{Conclusion and Perspective}
Beyond illustrating that certain properties of Macdonald polynomials can be proved through elementary manipulations on the Yang–Baxter graph, thereby recasting some difficult algebraic problems into combinatorial statements that are both more accessible and easier to resolve, our results constitute a step toward the proof of the conjectures of Bernevig and Haldane \cite{BH2008_2}. Here is the programme of study that we propose to pursue in the forthcoming papers :
\begin{itemize}
    \item By the use of Proposition \ref{prop-elemjumps}, Jump operations make it possible to construct Macdonald polynomials while avoiding the appearance of unnecessary intermediate poles , which are troublesome when one wishes to perform specializations. Using this strategy, we were able to show that staircase polynomials avoid certain poles that play a crucial role in specializations for which the Dunkl operators vanish, see {\it e.g.} \cite{Dunkl2019,CD2020}. These properties can be generalized by providing explicit formulas for the denominators (or, at worst, for polynomials divisible by these denominators) of Macdonald polynomials indexed by certain vectors. In particular, the case of partitions should admit a simple combinatorial formulation and would be especially interesting to understand, since it would allow one, via symmetrization, to derive information about the denominators of symmetric Macdonald polynomials, which, after degeneration of the parameters, yield the polynomials appearing in Haldane’s formulas \cite{BH2008_1,BH2008_2}.
    \item In \cite{CDL2018,CDL2022}, one of the authors, in collaboration with L. Colmenarejo and C. F. Dunkl, derived explicit formulas for quasi-staircase Macdonald polynomials. These formulas express the considered Macdonald polynomials, for some specialization of the variables and the parameters, in a completely factorized form, as generalized Vandermonde polynomials. For instance, one has
    \begin{equation}
    \begin{array}{l}
    M_{221100}(x_1,x_2,ty_1,y_1,ty_2,ty_2;t=v^2, q=v^{-3})=\\\displaystyle v^{45}(vy_2-y_1)(y_2-vy_1)(v^2y_2-y_1)(v^4y_2-y_1)\prod_{\genfrac{}{}{0pt}{2}{i,j\in\{1,2\}} {k\in\{1,4\}}}(v^ky_i-x_j).
    \end{array}
    \end{equation}
    Others formulas are also known \cite{MacdoforDummies}, allowing us to express Macdonald polynomials indexed by repeated staircase as similar products. For instance, one has
    \begin{equation}
        M_{210210}(\mathbf x;q=t^{-3})=t^6\prod_{1\leq i<j\leq 6}\left( t{\it x_j}-{\it x_i} \right) .
    \end{equation}
    Finally, numerical evidence suggests more general formulas involving Macdonald polynomials indexed by repeated staircases whose steps have heights greater than one. For instance, one has
    \begin{equation}
    \begin{array}{l}
    M_{420420}(x_1,ty_1,ty_2,x_2,y_1,y_2;t=v^2,q=v^{-3})=\\\displaystyle
    v^{45}(vy_1-y_2)(vy_2-y_1)(v^2y_2-y_1)^2(v^4y_1-x_1)(v^4y_2-x_1)\\\displaystyle\times(v^2y_1-x_2)(v^2y_2-x_2)\prod_{i,j\in\{1,2\}}(vy_i-x_j),
    \end{array}
    \end{equation}
    \begin{equation}
    \begin{array}{l}
        M_{4202020}(x,t^2y_1,t^2y_2,ty_1,ty_2,y_1,y_2;t=v, q=-v^{-2})=\\\displaystyle
        v^3(y_1+y_2)(vy_1+y_2)(vy_2+y_1)(vy_2-y_1)^3\prod_{\genfrac{}{}{0pt}{2}{i\in\{1,2\}}{\alpha\in\{-v,v^3\}}}(\alpha y_i-x)
    \end{array}\end{equation}
    and
    \begin{equation}
        \begin{array}{l}
        M_{630630}(x_1,ty_1,ty_2,x_2,y_1,y_2;t=v, q=e^{\frac{2\i\pi} 3}v^{-1})=\\\displaystyle
        (vy_2-y_1)^2\prod_{\genfrac{}{}{0pt}{2}{i,j\in\{1,2\},\ i\neq j}{\alpha\in\{1,v\}}}(\alpha e^{\frac{2\i\pi} 3}y_i-y_j)\times\\\displaystyle\times
        \prod_{\genfrac{}{}{0pt}{2}{i,j\in\{1,2\}}{\alpha\in\{1,v\}}}(\alpha e^{\frac{2\i\pi}3} y_i-x_j)
        \prod_{i\in\{1,2\}}(v^2y_i-x_1)(vy_i-x_2).
        \end{array}
    \end{equation}
    
    We expect to obtain unified formulas for these families, together with elementary combinatorial proofs based on the Yang–Baxter graph.
    \item Finally, our objective is to use the tools developed here to give proofs, as elementary and combinatorial as possible, of the two remaining conjectures of Bernevig and Haldane.
\end{itemize}
Other notions introduced in this paper call for further investigation. For instance, Proposition \ref{prop-elemjumps} suggests that one may focus on subgraphs of the Yang-Baxter graph and explore their algebraic properties. To illustrate these ideas, let us consider \(N=3k\) variables and the family of Macdonald polynomials indexed by vectors of the form
$
[a_1,a_1,a_1,a_2,a_2,a_2,\ldots,a_k,a_k,a_k].
$ 
One can extract from the Yang-Baxter graph a subgraph that is isomorphic to the Yang-Baxter graph for $k$ variables. In particular, the vertices
$
[a_1,a_1,a_1,a_2,a_2,a_2,\ldots,a_k,a_k,a_k]$
correspond to the vertices $[a_1,a_2,\ldots,a_k]$, the jumps $J^*_{3m+1,k}$ correspond to the non-affine edges 
$\Yang{?}{m+1}$ (the precise correspondence remains to be established), and the (triple) affine actions $\affineop^3$ correspond to the (single) affine action $\affineop$.
This raises the question of how this graph should be interpreted within the framework of DAHA representations, and of its relation to partially symmetric Macdonald polynomials \cite{Goodberry2024}.

\subsubsection*{Acknowledgments}
We acknowledge support from the Plan France 2030 through the project
ANR-22-PETQ-0006. 

\def\doi#1{doi: \href{https://doi.org/#1}{#1}}
\def\isbn#1{isbn: #1}


\begin{thebibliography}{99}
\bibitem{BH2008_1} B.A.~Bernevig and F.D.M.~Haldane,  Model Fractional Quantum Hall States and Jack Polynomials, \emph{Phys. Rev. Lett.}  {\bf 100} (2008), 246802  \doi{10.1103/PhysRevLett.100.246802}
\bibitem{BH2008_2}B.A.~Bernevig and F.D.M.~Haldane, Generalized clustering conditions of Jack polynomials at negative Jack parameter $\alpha$, \emph{Phys. Rev. B}  \textbf{77} (2008), 184502 \doi{10.1103/PhysRevB.77.184502}
\bibitem{Cherednik1995}
I.~Cherednik,
Double affine Hecke algebras, Knizhnik--Zamolodchikov equations, and Macdonald's operators,
\emph{Int. Math.s Res. Not.} {\bf 9} (1992) 171--180 \doi{10.1155/S1073792892000199} 
\bibitem{CDL2018}  L.~Colmenarejo, C.F.~Dunkl and J.-G.~Luque, Factorizations of Symmetric Macdonald Polynomials, \emph{Symmetry}
 \textbf{10}(11), 541 (2018) \doi{10.3390/sym10110541}
\bibitem{CD2020} L.~Colmenarejo and C.F.~Dunkl, Singular Nonsymmetric Macdonald Polynomials
and Quasistaircases, \emph{SIGMA}  {\bf 16} (2020), 010 \doi{10.3842/SIGMA.2020.010}
\bibitem{CDL2022}  L.~Colmenarejo, C.F.~Dunkl and J.-G.~Luque,  Connections between vector--valued
and highest weight Jack and Macdonald polynomials, \emph{Ann. Inst. H. Poincaré D
Comb. Phys. Interact.} {\bf 9} (2022), 297--348  \doi{10.4171/AIHPD/119}
\bibitem{Dunkl2019} C.F.~Dunkl, {\it Some Singular Vector-Valued Jack and Macdonald Polynomials}, \emph{Symmetry}  {\bf 11} (2019), 503 \doi{10.3390/sym11040503} 
\bibitem{FKLW2003} M.~Freedman, A.~Kitaev, M.~Larsen, and Z.~Wang,  Topological quantum computation, \emph{Bull. Amer. Math. Soc.} {\bf 40} (2003), no 1, 31--38 \doi{10.1090/s0273-0979-02-00964-3}
\bibitem{Goodberry2024} B.~Goodberry,  Type $A$ partially--symmetric Macdonald polynomials, \emph{Alg. Comb.} \textbf{7} (2024), no 6, 1647--1694 \doi{10.5802/alco.388} 
\bibitem{Iwahori1965}
N.~Iwahori and H.~Matsumoto,
  On some Bruhat decomposition and the structure of the Hecke rings of $p$-adic Chevalley groups,
 \emph{Publications Mathématiques de l'IHÉS} \textbf{25} (1965), 5--48 \doi{10.1007/BF02684396}
\bibitem{KazLus1979}
D.~Kazhdan and G.~Lusztig,
 Representations of Coxeter groups and Hecke algebras,
 \emph{Invent. Math.} \textbf{53} (1979), 165--184  \doi{10.1007/BF01390031}
 \bibitem{Kitaev1997}A.~Kitaev, Fault-tolerant quantum computation by anyons, \emph{Annals of Physics} \textbf{303} (1997), no 1, 2--30 \doi{10.1016/S0003-4916(02)00018-0}
 \bibitem{Lascoux2001}
A.~Lascoux, Yang–Baxter graphs, Jack and Macdonald polynomials,
 {\rm Annals of Combinatorics} \textbf{5} (2001), 397--424 \doi{doi.org/10.1007/s00026-001-8019-3}
 \bibitem{MacdoforDummies} 
A.~Lascoux,
{\it Schubert and Macdonald Polynomials, a Parallel},
 Manuscript, available online at \url{https://people.smp.uq.edu.au/OleWarnaar/pubs/MacSchub.pdf}
\bibitem{Lascoux2013} A.~Lascoux,
Polynomial Representations of the Hecke Algebra of the Symmetric Group, \emph{Int. J. Algebra Comput.} \textbf{23} (2013), no 4, 803--818 \doi{10.1142/S0218196713400109}
 \bibitem{Macdonald1988} I.G.~Macdonald,  A new class of symmetric functions, \emph{ Séminaire Lotharingien de Combinatoire} \textbf{20} (1988),  B20a \url{https://www.mat.univie.ac.at/~slc/opapers/s20macdonald.html}
\bibitem{Macdonald1995}  I.G.~Macdonald, \emph{Symmetric functions and Hall polynomials}, Second edition, Oxford Mathematical Monographs. Oxford Science Publications. The Clarendon Press, Oxford University Press, New York, 1995. x+475 pp.
\bibitem{Macdonald2000} I.G.~Macdonald, Orthogonal polynomials associated with root systems, \emph{Séminaire Lotharingien de Combinatoire} \textbf{45}(2000--2001) B45a 
\url{https://www.mat.univie.ac.at/~slc/wpapers/s45macdonald.html}
\bibitem{NSSFD2008}
C.~Nayak, S.~H.~Simon, A.~Stern, M.~Freedman, and S.~Das~Sarma,
Non-Abelian anyons and topological quantum computation,
 \emph{Rev. Mod. Phys.} {\bf 80} (2008), no  3, 1083--1159 \doi{10.1103/RevModPhys.80.1083}
\end{thebibliography}
\end{document}